\documentclass[english, a4paper]{smfart}

\usepackage[utf8]{inputenc}
\usepackage[T1]{fontenc}
\usepackage{babel}
\usepackage[babel]{csquotes}
\usepackage{verbatim}

\usepackage{amsmath}
\usepackage{amsfonts}
\usepackage{amsthm}
\usepackage{bbm}
\usepackage{amssymb}
\usepackage{stmaryrd}

\usepackage{float}
\usepackage{tikz}
\usetikzlibrary{shapes.misc}
\tikzset{cross/.style={cross out, draw, 
         minimum size=2*(#1-\pgflinewidth), 
         inner sep=0pt, outer sep=0pt}}
\usepackage{IEEEtrantools}
\usepackage{url}
\usepackage{enumitem}
\usepackage{hyperref}

\mathcode`l="8000
\begingroup
\makeatletter
\lccode`\~=`\l
\DeclareMathSymbol{\lsb@l}{\mathalpha}{letters}{`l}
\lowercase{\gdef~{\ifnum\the\mathgroup=\m@ne \ell \else \lsb@l \fi}}
\endgroup

\newtheorem{theorem}{Theorem}[section]
\newtheorem{proposition}[theorem]{Proposition}
\newtheorem{lemma}[theorem]{Lemma}

\newtheorem{conjecture}[theorem]{Conjecture}

\theoremstyle{definition}
\newtheorem*{definition}{Definition}

\newtheorem*{example}{Example}
\newtheorem*{remark}{Remark}

\newcommand{\norm}[1]{\lVert#1\rVert}

\usepackage{color}

\newcommand{\zhizhong}[1]{{\color{green} \sf $\clubsuit\clubsuit\clubsuit$ Zhizhong : [#1]}}

\title{Local distribution of rational points in flag varieties}
\author{Zhizhong Huang}
\address{State Key Laboratory of Mathematical Sciences, Academy of Mathematics and Systems Science, Chinese Academy of Sciences, Beijing 100190, China}
\email{zhizhong.huang@yahoo.com}
\author{Nicolas de Saxcé}
\address{CNRS -- Université Sorbonne Paris Nord,
	LAGA / UMR 7539,
	93430 Villetaneuse, France}
\email{desaxce@math.univ-paris13.fr}

\begin{document}

\begin{abstract}
Given a flag variety $X$ defined over $\mathbb{Q}$ and a point $x$ in $X(\mathbb{R})$, we study approximations to $x$ by points $v$ in $X(\mathbb{Q})$, and show that, with an appropriate rescaling, those approximations equidistribute when $x$ is chosen randomly according to a Lebesgue measure on $X(\mathbb{R})$, or when $x$ is defined over $\mathbb{Q}$ and satisfies some non-degeneracy condition.
\end{abstract}

\setcounter{tocdepth}{1}
\maketitle
\tableofcontents

\section{Introduction}

\subsection*{Counting rational points and equidistribution}

Let $X$ be an algebraic variety defined over $\mathbb{Q}$. The study of the set $X(\mathbb{Q})$ of rational points on $X$ is a central question in number theory that strongly depends on the geometric and arithmetic properties of $X$.  When $X$ is of Fano type and such that $X(\mathbb{Q})$ is Zariski dense, Manin and his collaborators \cite{Batyrev-Manin,fmt} put forward conjectures about the asymptotic behavior of the number of rational points with bounded anticanonical height as the bound goes to infinity. This gives a quantitative measure of the density of $X(\mathbb{Q})$.

Starting with the work of Franke--Manin--Tschinkel~\cite{fmt}, the case where $X$ is a generalized flag variety --- i.e. a quotient \( X=G/P \) of a semisimple algebraic \( \mathbb{Q} \)-group $G$ by a parabolic \( \mathbb{Q} \)-subgroup $P$ --- was much studied, and asymptotic behavior of the counting function
\[
N(X,H) = \#\{v\in X(\mathbb{Q})\ :\ \mathrm{H}(v)\leq H\},
\] when $\mathrm{H}$ is a height associated to the anticanonical line bundle, was shown to be $$N(X,H) \sim_{H\to+\infty} c_{X}H(\log H)^{r-1},$$ where $c_X>0$ is a constant that has received an interpretation by Peyre \cite{Peyre}, and $r$ is the rank of the Picard group $\operatorname{Pic}(X)$.
Furthermore, Peyre~\cite[Corollaire~6.2.17]{Peyre} showed that, in the case where $G$ is quasi-split, the counting measures
\begin{equation}\label{eq:counting}
\mu_{X,H} = \frac{1}{N(X,H)}\sum_{\substack{v\in X(\mathbb{Q})\\ \mathrm{H}(v)\leq H}} \delta_v
\end{equation}
converge as $H\to +\infty$ to a natural measure $\mu_X$ on $X(\mathbb{R})$, called \emph{Tamagawa measure}.
Peyre's equidistribution result implies in particular that if $\mathbf{B}_X$ is a fixed open ball in $X(\mathbb{R})$, the number $N(\mathbf{B}_X,H)$ of rational points of height at most $H$ inside $\mathbf{B}_X$ satisfies
\begin{equation}\label{eq:countbx}
N(\mathbf{B}_X,H) = N(X,H)\cdot \mu_X(\mathbf{B}_X) \cdot (1+o(1))
\quad\mbox{as } H \to +\infty,
\end{equation}
with an error term depending \emph{a priori} on $\mathbf{B}_X$.
More recently, it was shown by Mohammadi and Salehi Golsefidy \cite[Theorem~4]{msg} (see also Batyrev--Tschinkel \cite[Theorem 4.4.1]{Batyrev-Tschinkel} and Thunder \cite[III.2]{Thunder}) that if $G$ is almost $\mathbb{Q}$-simple and the height function $\mathrm{H}$ is associated to any ample line bundle $L$, there exist constants $a>0$, $b\geq 0$ and $c_{X,L}>0$ depending on $L$ such that
\[
N(X,H) \sim_{H\to+\infty} c_{X,L}\cdot H^a \cdot (\log H)^b.
\]

\subsection*{Zooming measures and Diophantine exponents}
We wish to study to what extent the asymptotic equivalent \eqref{eq:countbx} remains valid when the radius of the ball $\mathbf{B}_X$ is allowed to shrink as $H$ tends to infinity.
For that, we fix a  point $x$ in $X(\mathbb{R})$, fix a distance function $\mathrm{d}$ around $x$, and rescale the counting measures restricted to a small ball $\mathbf{B}_X(x,r)$ centered at $x\in X(\mathbb{R})$ of radius $r\to 0$ as $H\to\infty$; this yields the notion of \emph{zooming measures}, introduced by the first author in his thesis~\cite{Huangthesis}.
For the sake of completeness, let us briefly recall the construction of those zooming measures.
Write $T_xX$ for the tangent space to $X(\mathbb{R})$ at $x$.
We fix a real neighborhood $U_x$ of $x$ and a smooth map
\[
p_x\colon  U_x  \to  T_xX
\]
such that $p_x(x)=0$ and $T_xp_x = \mathrm{Id}_{T_xX}$.
%
\begin{definition}[Zooming measures -- single height and Riemannian metric]
For $x\in X(\mathbb{R})$, $\tau>0$, and $t\geq 0$ a large parameter, the \emph{zooming measure} with zoom factor $\tau$ at time $t$ is the measure on $T_xX$ defined by
\[
	\mu_{x,\tau,t} = \sum_{\substack{v\in X(\mathbb{Q})\cap U_x\\ \mathrm{H}(v)\leq e^t}} \delta_{e^{\tau t} \cdot p_x(v)}.
\]
\end{definition}
\begin{remark}
A different choice of local chart $p_x$ results in a change of coordinates of the zooming measures $\mu_{x,\tau,t}$, but provided the conditions \( p_x(x)=0 \) and \( T_xp_x = \mathrm{Id}_{T_xX} \) are satisfied, it will not affect our results on their asymptotic behavior.
\end{remark}

For a continuous function \(f\) with compact support on \(T_xX\), the integral against the zooming measure is given by
\[
	\int f \,d\mu_{x,\tau,t} = \sum_{\substack{v \in X(\mathbb{Q})\cap U_x\\ \mathrm{H}(v) \leq e^{t}}} f(e^{\tau t}\cdot p_x(v)).
\]
Since \(f\) has compact support, the only rational points that appear in the sum on the right-hand side are those satisfying \(\mathrm{H}(v)\leq e^t\) and \(\operatorname{d}(v,x) \lesssim e^{-\tau t}\): the zooming measure describes the distribution of rational points of height at most \(e^t\) at distance at most \(e^{-\tau t}\) from the point \(x\).
Recall that the \emph{Diophantine exponent} of a point $x\in X(\mathbb{R})$ is defined as
\[
\beta(x)=\sup\left\{\beta\geqslant 0:\begin{aligned}
		&\exists \text{ a sequence } (v_i)\subset X(\mathbb{Q}) \text{ such that }\\ 
		&\lim_{i\to+\infty} \operatorname{d}(v_i, x)=0 \text{ and } \forall i,\ \operatorname{d}(v_i,x)\leq H(v_i)^{-\beta}
	\end{aligned}\right\}.
\]
Observe that if $\tau > \beta(x)$, then for all \( t>0 \) large enough, the ball $\mathbf{B}_X(x,e^{-\tau t})$ contains no rational point $v$ of height at most \( e^t \), except possibly the point $x$ itself.
The study of the zooming measures $\mu_{x,\tau,t}$ is therefore trivial in that range.

\subsection*{Main results}

At the origin of the present paper was the desire to understand zooming measures when \( \tau < \beta(x) \) on a flag variety $X$, given as a quotient $X=G/P$ of a semisimple algebraic $\mathbb{Q}$-group $G$ by a parabolic $\mathbb{Q}$-subgroup $P$.
A study of Diophantine approximation on such varieties was started in \cite{saxce_hdr}, where it was shown in particular that there exists a constant \( \beta_X \) such that \( \beta(x) = \beta_X \) for Lebesgue almost every \( x \) in \( X(\mathbb{R}) \).
In general, the Diophantine exponent of a point \( x \) is closely related to its position with respect to the rational Schubert subvarieties in \( X \).
Recall from \cite[\S8.5]{Springer} that \( X \) admits a decomposition into Schubert cells
\[
X=\bigsqcup_{w\in W/W\cap P} BwP,
\]
where $W$ is the Weyl group associated to \( G \) and a maximal \( \mathbb{Q} \)-split torus contained in \( P \), and $B$ is a minimal parabolic group contained in $P$.
A \emph{Schubert variety} is a subset of the form \(g\overline{BwP}\), with \(w\in W/W\cap P\) and \(g\in G\); it is said to be \emph{rational} if the element \( g \) can be taken in \( G(\mathbb{Q}) \).

Our main result describes the asymptotic behavior of the zooming measures when the point $x$ is generic in the sense of the Lebesgue measure, or algebraic and outside of any rational Schubert variety. 
In the particular case where the parabolic subgroup \(P\) defining the flag variety \(X= G/P\) 
has abelian unipotent radical, our results can be formulated in terms of the elementary definition of the zooming measures given above.
Below and in the rest of the paper, by abuse of notation, we write \( \overline{\mathbb{Q}} \) to denote the subfield of \( \mathbb{R} \) consisting of real algebraic numbers over \( \mathbb{Q} \).

\begin{theorem}[Generic local distribution]
\label{th:maini}
Let $G$ be a semisimple $\mathbb{Q}$-group, $P$ a $\mathbb{Q}$-parabolic subgroup with abelian unipotent radical and write $X=G/P$ for the associated flag variety.
There exist positive constants \(c_X\) and \(d_X\) such that the following holds.\\
For all $\tau\in(0,\beta_X)$, for all $x$ in $X(\overline{\mathbb{Q}})$ not contained in any proper rational Schubert variety, for all $f\in \mathcal{C}_c(T_xX)$,
\[
	\mu_{x,\tau,t}(f) \sim_{t\to+\infty} c_X \cdot e^{t d_X (\beta_X -\tau) } \cdot \left(\int_{T_xX} f\,d\mathrm{m}\right),
\] where $\mathrm{m}$ is the Lebesgue measure on $T_xX$. 
The same estimate holds for Lebesgue almost every $x$ in $X(\mathbb{R})$.
\end{theorem}

Recall that the $\mathbb{Q}$-rank of the flag variety $X$ is equal to $\mathrm{rk}_{\mathbb{Q}} X = \mathrm{rk}_{\mathbb{Q}} G-\mathrm{rk}_{\mathbb{Q}} P$.
The assumption that the unipotent radical of \( P \) is abelian implies in particular that \( P \) is maximal and \( X \) is a flag variety of \( \mathbb{Q} \)-rank one.
We also study local distribution of rational points on general flag varieties, without any restriction on the \( \mathbb{Q} \)-rank.
In that setting, one needs to adjust the definition of the zooming measures in two ways: first the distance used on $X$ should be the Carnot-Carathéodory metric (and so the rescaling on $T_xX$ is not by homotheties), and second, when the variety is not of rank one, one has to control simultaneously all heights in a generating family of heights.
This idea that controlling all the heights can help avoid degenerate behaviors originates in Peyre~\cite[\S4]{Peyre2}.
In fact, as a by-product of our approach, we will also be able to answer affirmatively \cite[Question~4.8]{Peyre2} in the particular setting of flag varieties.
Our main result on local distribution on general flag varieties, Theorem~\ref{th:gldg} in the main body of the text, gives an asymptotic formula for zooming measures similar to the one in Theorem~\ref{th:maini}, but only when the zoom factor \( \tau \) is small enough.
It would be interesting to determine in that setting what the optimal range is for the zoom factor for such an asymptotic equivalent to hold.
The reader is referred to the Section \ref{se:openprob} for a more thorough discussion of this problem and other related questions.

\bigskip

Before we briefly discuss the proofs of our results, let us illustrate Theorem~\ref{th:maini} with the most elementary examples.

\subsection*{Examples}
\subsubsection*{Projective spaces}
Let $X=\mathbb{P}^n$ be the projective $n$-space.
Then $X\simeq \operatorname{SL}_{n+1}/P$ where $P$ is the stabilizer of a rational line in the standard representation.
It is well known from metric Diophantine approximation that the generic Diophantine exponent in that is equal to \( \beta_{\mathbb{P}^n} = \tfrac{n}{n-1} \).
For the projective space, Theorem \ref{th:maini} can also be proved with a more direct approach; this is done for instance in \cite{Huangthesis} when $n=1$, using an argument combining Roth's theorem for real algebraic numbers and the theory of uniform distribution modulo one.

\subsubsection*{Projective quadrics}
Let $q$ be a non-degenerate indefinite rational quadratic form in $n\geqslant 4$ variables and let $G=\operatorname{SO}_q\subset\operatorname{GL}_n$.
Let $e_0\in \mathbb{Q}^n$ be such that $q(e_0)=0$ and $P=\mathrm{Stab}_G(\mathbb{Q}e_0)$.
Consider the projective quadric hypersurface $X=(q=0)\simeq G/P$.
For those varieties Fishman, Kleinbock, Merrill and Simmons \cite{FKMS} proved that $\beta_X=1$. 

In \cite{Kelmer-Yu,HSS1,HSS2}, it was shown that the zooming measures at \emph{any} point \( x \) equidistribute, provided that the zoom factor $\tau\in(0,\tfrac{1}{2})$.
This zoom factor is optimal when $x$ is a $\mathbb{Q}$-point.
Our approach works with an arbitrary zoom factor \( \tau \in (0,1) \), provided that the point \( x \) is generic:  \( x \) is either a random point for the Lebesgue measure, or an element of \( X(\overline{\mathbb{Q}}) \) not contained in any totally isotropic rational subspace for the quadratic form \( q \).

\subsubsection*{Grassmannians}
The Grassmann variety $\mathrm{Gr}_{l,d}$ parametrizing $l$-dimensional subspaces within a given $d$-dimensional space is a flag variety under $G=\operatorname{SL}_d$ with the parabolic subgroup \( P \) formed by the matrices whose lower-left $(l-d)\times l$ entries are all zero.
In that setting, it was shown by the second author \cite{deSaxcegrass} that $\beta_X = \tfrac{d}{l(d-l)}$.
The condition that a point \( x \) is not contained in any proper rational Schubert variety is equivalent to saying that every rational subspace \( W \) in \( \mathbb{R}^d \) intersects \( x \) minimally : \( \dim(x\cap W) = \max(0, \dim x + \dim W -d) \).

\bigskip
It is not difficult to reformulate Theorem~\ref{th:maini} as a counting statement for integer points in the cone \( \tilde{X} = G/L \) over the variety \( X \), where $L$ is the subgroup of $P$ consisting of elements lying in the kernel of every $\mathbb{Q}$-character of $P$.
However, the region arising from applying the zoom measures to a characteristic function of a ball in the tangent plane is not \enquote{well-rounded}: its boundary can be large compared to its interior, especially if the zoom factor is large.
Therefore, one cannot directly argue that the number of integer points inside it is comparable with its volume.
To resolve this issue, we make use of the action of a well-chosen diagonalisable flow $(g_t^x)_{t \in \mathbb{R}}$, which turns the lopsided region into a well-rounded one \( \mathcal{R}_t \).
One then has to count points in \( \mathcal{R}_t \) that belong to a new lattice \( \Delta_t^x \), image of the integer lattice under the flow \( g_t^x \).
Fortunately, when \( x \) is chosen randomly according to the Lebesgue measure on \( X(\mathbb{R}) \), or is algebraic and not contained in any rational Schubert variety, the effect of \( g_t^x \) on the successive minima of the lattice \( \Delta_t^x \) is negligible.
This can be used to show that the lattice point counting is indeed comparable with the volume.
For that, we combine the strategy of Mohammadi and Salehi Golsefidy \cite{msg}, which goes back to \cite{DRS,EM}, and more recent results on effective equidistribution of periodic orbits in finite-volume homogeneous spaces~\cite{Dabbs-Kelly-Li, shi}.
We note that, inspired by our approach, Pfitscher \cite[Theorem 1.6]{Pfitscher} has recently obtained Schmidt-type counting results for rank-one flag varieties.

\section{Heights and counting on a flag variety}
\label{sec:hc}

Let $X=G/P$ be a rational flag variety, where $G$ is a connected semisimple algebraic group over $\mathbb{Q}$ and $P$ is a parabolic subgroup over $\mathbb{Q}$. 
If
\(
\iota\colon\widetilde{G}\to G
\)
is an isogeny, then the inverse image $\iota(P)$ is equal to the parabolic subgroup $\tilde{P}$ with Lie algebra $\mathfrak{p}$,
so
\(
X = \tilde{G}/\tilde{P}.
\)
This will allow us to reduce to the case where \( G \) is simply connected.
Moreover, by the classification of \( \mathbb{Q} \)-parabolic subgroups given in~\cite[Théorème~11.8]{borel_iga}, any \( \mathbb{Q} \)-anisotropic simple \( \mathbb{Q} \)-factor of \( G \) must be included in \( P \).
Quotienting by the sum of all \( \mathbb{Q} \)-anisotropic factors, we may always replace \( G \) by a subgroup \( G_1 \) containing \( P_1 = P \cap G_1 \) as a parabolic subgroup, so that \( X = G/P = G_1/P_1 \).
Note that since any (real) compact factor is included in a \( \mathbb{Q} \)-anisotropic factor, the group \( G_1 \) has no compact factor.
Thus, we shall from now on assume without loss of generality that \( G \) is simply connected and without compact factors.

We let $T\subset P$ be a maximal split  $\mathbb{Q}$-torus.
Write $A=T(\mathbb{R})^\circ$ for the connected component of $T(\mathbb{R})$, $\mathfrak{a}$ for the Lie algebra of $A$, and $\Pi\subset\mathfrak{a}^*$ for a basis of the root system associated to $G$ and $T$.
Again by the classification of parabolic subgroups over $\mathbb{Q}$ \cite[Théorème~11.8]{borel_iga}, there exists a subset $\theta\subset\Pi$ such that $\mathfrak{p}$, the Lie algebra of $P$, is equal to the direct sum of $\mathfrak{a}$ with all spaces $\mathfrak{g}_\beta$ associated to the roots $\beta$ in which no element of $-\theta$ appears:
\[
\mathfrak{p} = \mathfrak{a}\oplus \bigoplus_{\substack{\beta:\forall\alpha\in\theta,\\ \beta\not\succ -\alpha}}\mathfrak{g}_\beta.
\]
The unipotent radical of \( P \) and its opposite will be denoted $U$ and $U^-$, with Lie algebras \( \mathfrak{u} \) and \( \mathfrak{u}^- \) respectively:
\[
\mathfrak{u} = \bigoplus_{\substack{\beta:\exists\alpha\in\theta\\ \beta\succ\alpha}}\mathfrak{g}_\beta
\quad\mbox{and}\quad
\mathfrak{u}^- = \bigoplus_{\substack{\beta:\exists\alpha\in\theta\\ \beta\succ\alpha}}\mathfrak{g}_{-\beta}.
\]
(We recall that the notation $ \beta\succ\alpha$ means that $\beta-\alpha$ is a sum of simple positive roots with positive coefficients.)

\subsection{Representations and heights on $X$}
	
We shall denote by \((\varpi_{\alpha})_{\alpha\in \Pi}\) the family of fundamental \(\mathbb{Q}\)-weights of \(G\).
Those are defined by
\[
	\forall \alpha , \beta \in \Pi,\quad
	\frac{2\langle\varpi_{\alpha},\beta\rangle }{\langle\beta,\beta\rangle}
	=\delta_{ \alpha , \beta }.
\]
Since $G$ is simply connected, it follows from~\cite[\S12.13]{boreltits} that for each fundamental weight \( \varpi_\alpha \), there exists a strongly irreducible $\mathbb{Q}$-representation $G\to\mathrm{GL}_{V_{ \varpi_\alpha}}$ with highest weight $ \varpi_\alpha$.
The weight space in $V_{ \varpi_\alpha}$ associated to $ \varpi_\alpha$ is a line generated by a rational vector $e_{ \varpi_\alpha}$.
For the study of the flag variety \( X = G/P \), we shall only be interested in the fundamental weights $\varpi_\alpha$ that are trivial on $P$, i.e., $\alpha\notin \theta$.
For \(\alpha \not\in \theta\), the line generated by $e_{ \varpi_\alpha}$ is stable under $P$ and we therefore obtain an embedding
\[
	\begin{array}{llll}
		\iota_\alpha \colon & G/P & \hookrightarrow & \mathbb{P}(V_{ \varpi_\alpha}) \\
			    & gP  & \mapsto & [g e_{ \varpi_\alpha}]
	\end{array}.
\] 
Let $\mathcal{P}$ denote the set of primes numbers. Any choice of an adelic norm $(\lVert \cdot \rVert_p)_{p\in \mathcal{P}\cup \infty}$ on \(V_{\varpi_\alpha}\) yields a height \( \mathrm{H}_\alpha \) on \(X(\mathbb{Q})\) given by
\[
	\mathrm{H}_\alpha(gP) = \prod_{p \in \mathcal{P}\cup \infty} \lVert g e_{\varpi_\alpha}\rVert_p.
\]
Equivalently, one may consider the lattice associated to the adelic norm
\[
	\Lambda_\alpha = \{ \mathbf{v} \in V_{\varpi_\alpha }(\mathbb{Q})\ |\ \forall p\in \mathcal{P},\ \lVert \mathbf{v}\rVert_p \leq 1\}	
\]
and then \(\mathrm{H}_{\alpha }(v)\) is the Euclidean norm of a primitive element \(\mathbf{v} \in \Lambda_\alpha\) proportional to \(\iota_\alpha(v)\).
In the sequel, we fix a maximal compact subgroup \( K \) in \( G(\mathbb{R}) \) and always assume that the Euclidean norm on \( V_{\varpi_\alpha } \) is invariant under \( K \).
It will also be convenient to use the logarithmic height
\[
	\mathrm{h}_\alpha (v) = \log \mathrm{H}_\alpha(v) 
		= \sum_{p\in \mathcal{P}\cup\infty} \log\ \lVert g e_{\varpi_\alpha} \rVert_p.
\]
In~\cite[\S4]{Peyre2}, Peyre suggested to study simultaneously all possible heights on a nice Fano-type variety by defining the multiheight of a rational point on \( X \).
In our setting of flag varieties, this boils down to the following definition.
\begin{definition}[Multiheight]
	The \emph{multiheight} of a point \( v \) in \( X( \mathbb{Q}) \) is the unique element \( \mathbf{h}(v) \) in \(\mathfrak{a}_\theta = \mathfrak{a} \cap \theta^\perp\) such that for each \( \alpha \in \Pi\setminus\theta \),
	\[
		\langle \varpi_\alpha , \mathbf{h}(v) \rangle = \mathrm{h}_\alpha(v).	
	\]
\end{definition}

Since the lattice \( \Lambda_\alpha \) is discrete in \(  V_{\varpi_\alpha } (\mathbb{R}) \), each height \( \mathrm{h}_\alpha(v) \) is uniformly bounded below when \( v \) varies among all rational points in \( X \), and therefore, the multiheight \( \mathbf{h}(v) \) remains within bounded distance of the \emph{dual effective cone} \( C_{\mathrm{eff}}^\vee \), defined by
\[
	C_{\mathrm{eff}}^\vee = \{ Y \in \mathfrak{a}_{ \theta} \ |\ \forall \alpha \in \Pi\setminus \theta,\ \langle \varpi_\alpha , Y \rangle \geq 0 \}.
\]
We consider the measure \(\nu\) on \( \mathfrak{a}_{ \theta}\) defined by
	\[
	\nu( \mathcal{D}) = \int_{ \mathcal{D}} e^{ \langle \varrho_X, y\rangle} \,dy,
	\]
where \(\varrho_X\) denotes the sum of all roots in the unipotent radical $U$ of \(P\), counted with multiplicities, \(dy\) is some choice of Lebesgue measure on \(\mathfrak{a}_\theta \), and \( \mathcal{D} \) is any measurable subset of \( \mathfrak{a}_\theta \).
Following Peyre's suggestions \cite[\S4]{Peyre}, we want to study rational points in $X$ whose multiheight belongs to a compact subset of large measure in \( C_{\mathrm{eff}}^\vee\). 
For that, for any measurable subset $\mathcal{D}\subset \mathfrak{a}_\theta$, we define \begin{equation}\label{eq:hD}
	X(\mathbb{Q})_{\mathbf{h} \in \mathcal{D}}= \{v\in X(\mathbb{Q}): \mathbf{h}(v)\in \mathcal{D}\}.
\end{equation} 	
As suggested in \cite[Remark 4.6]{Peyre2}, the following growing family $(\mathcal{D}_t)_{t\to\infty}$ is of particular interest. We fix a compact domain \( \mathcal{D}_0 \subset \mathfrak{a}_{ \theta}\) with smooth boundary, an element \(u\) in the interior \(\mathring{C}_{\mathrm{eff}}^\vee\) of the dual effective cone, and let, for all \(t>0\),
\[
	\mathcal{D}_t = \mathcal{D}_0 + t\cdot u.
\]
Our aim is to prove an asymptotic formula for the cardinality of $X(\mathbb{Q})_{\mathbf{h} \in \mathcal{D}_t}$, as \(t\) tends to \(+\infty\).


\begin{example}[Polyhedrons and bounds on heights]
	For each \(\alpha \in \Pi\setminus \theta \), fix two positive parameters \(a_\alpha < b_\alpha \).
	Let \(u\) be the unique element in \(\mathfrak{a}_\theta \) such that \( \langle \varpi_\alpha, u \rangle = 1\) for all \(\alpha \) in \(\Pi\setminus \theta \), and consider the polyhedron
	\[
		\mathcal{D}_0 = \{ z \in \mathfrak{a}_{\theta}\ |\ \forall \alpha \in \Pi\setminus \theta,\ \log a_\alpha \leq \langle \varpi_\alpha, z \rangle \leq \log b_\alpha\}.
	\]
	Letting \(H=e^t\), the condition \(\mathbf{h}(v) \in \mathcal{D}_t\) is equivalent to the system of inequalities on the heights:
	\[
		\forall \alpha \in \Pi\setminus \theta,\quad
		a_\alpha H \leq \mathrm{H}_\alpha(v) \leq b_\alpha H.
	\]
\end{example}


\begin{remark}
	For a general Fano-type variety \( X \), one first defines the \emph{effective cone} \( C_{\mathrm{eff}} \) in \( \mathrm{Pic}(X)_{\mathbb{R}} \) as the cone generated by effective divisors.
	The multiheight \( \mathbf{h}(v) \) of a rational point \( v \) in \( X(\mathbb{Q}) \) then takes values in its dual \( \mathrm{Pic}(X)_{\mathbb{R}}^\vee \), at bounded distance from the dual effective cone
	\[
	C_{\mathrm{eff}}^\vee = \{ y \in \mathrm{Pic}(X)_{\mathbb{R}}^\vee \ |\ \forall x \in C_{\mathrm{eff}},\ \langle x,y \rangle \geq 0 \}.
	\]
	The dual effective cone $C_{\mathrm{eff}}^\vee$ is endowed with the measure \( \nu \) given by \( \nu(\mathcal{D}) = \int_{\mathcal{D}} e^{ \langle \omega_X^{-1},y \rangle }\,dy \), where \( \omega_X^{-1} \in \mathrm{Pic}(X)\) denotes the class of the anticanonical line bundle.
	Peyre~\cite[Question~4.8]{Peyre2} gave a prediction for a possible asymptotic formula for the cardinality of $X(\mathbb{Q})_{\mathbf{h} \in \mathcal{D}_t}$ \eqref{eq:hD}, as \(t\) tends to \(+\infty\).
	
	When \( X = G/P \) is a flag variety, we have $\operatorname{rk}(\operatorname{Pic}(X))=\operatorname{rk}(\mathbf{X}^*(P))$, the latter being the group of $\mathbb{Q}$-characters of $P$ (cf. \cite[\S2]{fmt}). With the above notation, the cone of effective divisors can be identified with 
	\[
	C_{\mathrm{eff}}=\sum_{\alpha \in \Pi \setminus \theta} \mathbb{R}^+\varpi_\alpha\subseteq \mathrm{Pic}(X)_{\mathbb{R}}, 
	\]
	and the dual effective cone becomes
	\[
	C_{\mathrm{eff}}^\vee = \{ Y \in \mathfrak{a}_{ \theta} \ |\ \forall \alpha \in \Pi\setminus \theta,\ \langle \varpi_\alpha , Y \rangle \geq 0 \}.
	\]
	With these identifications, the anticanonical line bundle $\omega_X^{-1}$ can be identified with $\varrho_X$, i.e. the sum of all roots occurring in the unipotent radical $U$ of \(P\), with multiplicities equal to the dimension of the corresponding eigenspace \cite[(6.2.1)]{Peyre}.
	We refer the interested reader 
	to \cite[\S6]{Peyre} for the passage from that general setting to that of flag varieties.
	For our arguments, one may take as a definition the explicit description of the dual effective cone as a subset of \( \mathfrak{a}_\theta \).
\end{remark}

Before we turn to our study of zooming measures, we explain how the techniques of homogeneous dynamics, as developed in particular by Mohammadi and Salehi Golsefidy~\cite{msg}, can be used to check the validity of Peyre's formula in the case where \( X \) is a flag variety.
This is the content of Theorem~\ref{th:count} below, whose proof will be given in the next two paragraphs.

\begin{theorem}[Counting points with all heights controlled]
	\label{th:count}
	Let \(X=G/P\) be a flag variety, given as the quotient of a semisimple \(\mathbb{Q}\)-group by a parabolic \(\mathbb{Q}\)-subgroup.
	There exists a constant \(\kappa_X\) depending only on \(X\) such that for any choice of compact domain \(\mathcal{D}_0\) and \(u\) in \( \mathring{C}_{\mathrm{eff}}^\vee\)
	\[
		\# X(\mathbb{Q})_{\mathbf{h} \in \mathcal{D}_t}
		\sim_{t\to+\infty}
		\kappa_X \cdot \nu(\mathcal{D}_t).
	\]
\end{theorem}

\begin{remark}
	One has \( \nu(\mathcal{D}_t) = e^{ t \langle \varrho_X, u \rangle } \nu(\mathcal{D}_0) \).
\end{remark}

	In his conjectural formula, Peyre normalizes the Lebesgue measure on \( \mathfrak{a}_\theta \simeq \mathrm{Pic}_{ \mathbb{R}}(X)^\vee \) so that the dual of the Picard group has covolume one.
	This natural choice allows him to express the constant \( \kappa_X \) in terms of arithmetic and geometric data on \( X \).
	Our proof also yields a formula for \( \kappa_X \), see \eqref{eq:kappa} below, but it would require a more careful analysis to check directly that it indeed coincides with Peyre's constant.
	We briefly comment on this problem at the end of the paper.

\subsection{An arithmetic group acting on rational points}

Recall that for each \(\alpha \) in \(\Pi\setminus \theta \), the adelic norm on \(V_{\varpi_\alpha }\) defines a rational lattice \(\Lambda_\alpha \) in \(V_{\varpi_\alpha }\).
Applying~\cite[Proposition~7.12]{borel_iga} several times, we may construct an arithmetic subgroup \(\Gamma \) in \(G\) such that each \(\Lambda_\alpha \) is stable under the action of \(\Gamma \).
By \cite[Proposition~15.6]{borel_iga}, the set of rational points \(X(\mathbb{Q})\) is a finite union of \(\Gamma \)-orbits, so we may fix a finite set \(C\) in \(G(\mathbb{Q})\) such that
\begin{equation}\label{eq:union}
	X(\mathbb{Q}) = \bigsqcup_{c \in C} \Gamma c P.
\end{equation}
The following elementary lemma will allow us to understand the multiheight of a point \( v \) in \( X(\mathbb{Q}) \) in terms of the element \( \gamma \) in \( \Gamma \) such that \( v=\Gamma c P \).

\begin{lemma}
	\label{lm:halpha}
	For each \(\alpha \in \Pi\setminus \theta\) and \(c \in C\), there exists \(q_{c,\alpha} \in \mathbb{Q}\) such that the height of an element \(v\) in \(X(\mathbb{Q})\) written as \(v = \gamma c P\) for some \(\gamma\) in \(\Gamma \) and \(c \in C\) is given by
	\[
	\mathrm{h}_\alpha(v) = \log\, \lVert q_{c,\alpha}\gamma c e_{\varpi_\alpha }\rVert.
	\]
\end{lemma}
\begin{proof}
	Let \(q_{c,\alpha }\in\mathbb{Q}\) be such that \(q_{c,\alpha} c e_{\varpi_\alpha} \) is a primitive element in the lattice \(\Lambda_\alpha \).
	Then, for any \(\gamma \in \Gamma \), the vector \(q_{c,\alpha }\gamma c e_{\varpi_\alpha }\) is a primitive element of \(\Lambda_\alpha \) proportional to \(\iota_\alpha(v)\), and the formula follows.
\end{proof}

\subsection{Counting and equidistribution}
\label{ss:ce}

We now explain how the argument used in \cite{msg} to count points of bounded height in flag varieties can be used to derive Theorem~\ref{th:count}.
The main ingredient in the proof is an equidistribution result for translates of horospherical measures in the space of lattices that appeared as \cite[Theorem~1 (ii)]{msg} in the case where the ambient group \( G \) is almost simple.
As explained in \cite[Theorem 1.4]{shi}, such an equidistribution holds in a more general setting, and implies the following. 

\begin{theorem}[Equidistribution of translated periodic orbits]
	\label{th:equi}
	Let \( G \) be a semisimple \( \mathbb{Q} \)-group without compact factors and \( \Gamma = G(\mathbb{Z}) \) be an arithmetic subgroup.
	Let \( P \) be a parabolic \( \mathbb{Q} \)-subgroup defined by a subset \( \theta \subset \Pi \) and set
	\begin{equation*}
		L = \bigcap_{\alpha \in \Pi \setminus \theta }\mathrm{Stab}_G e_{\varpi_\alpha}.
	\end{equation*}
	Denote by $\mathrm{m}_{G/\Gamma}$ (resp. \( \mathrm{m}_{L/(\Gamma\cap L)} \)) the probability Haar measure on $G/\Gamma$ (resp. on \( L/(\Gamma\cap L) \)). For \( y \in \mathfrak{a}_\theta \), let \[
	\left\lfloor y \right\rfloor = \min_{\alpha \in \Pi \setminus \theta} \varpi_\alpha(y). 
	\]
Then  the translated measures \( (e^y)_* {\mathrm{m}_{L/(\Gamma\cap L)}} \) equidistribute in \( G/\Gamma \) as long as $\left\lfloor y \right\rfloor\to+\infty$, namely for every \( \varphi \in \mathcal{C}_c^\infty(G/\Gamma) \),
\[\int_{L/(\Gamma\cap L)} \varphi(e^y l \Gamma)\,d\mathrm{m}_{L/(\Gamma\cap L)}(l) \sim_{\left\lfloor y \right\rfloor\to+\infty} \int_{G/ \Gamma} \varphi d\mathrm{m}_{G/\Gamma}.\] 
\end{theorem}

In Section~\ref{sec:zoom}, when we study the asymptotic behavior of zooming measures, we shall state an effective version of this asymptotic.

\begin{proof}[Proof of Theorem~\ref{th:count}]
It follows from \eqref{eq:union} that it is enough to study rational points in each orbit \(\Gamma c P\), \(c \in C\).
So we fix \(c\) and count the number of rational points \(v \in \Gamma c P\) such that
\[
			\mathbf{h}(v) \in \mathcal{D}_t.
\]
Write \( v = \gamma c P = c \gamma^c P \), where \( \gamma^c = c^{-1} \gamma c \), and note that the elements of \(\{c \gamma^c P\mid\gamma\in \Gamma\} \) are in one-to-one correspondence with cosets in \(\Gamma^c /(\Gamma^c \cap P)\), where \( \Gamma^c = c^{-1}\Gamma c \).
From Lemma~\ref{lm:halpha}, one has for each \(\alpha \) in \(\Pi \setminus \theta \),
\[
	\mathrm{h}_\alpha(c \gamma^c P) = \log \lVert q_{c,\alpha }c \gamma^c e_{\varpi_\alpha} \rVert.
\]
For \(g\) in \(G\), define \( \mathbf{h}^{(c)}(g) \in \mathfrak{a}_\theta \) as the unique element such that
\[
	\forall \alpha \in \Pi \setminus \theta, \quad \langle \varpi_\alpha, \mathbf{h}^{(c)}(g) \rangle = \log \lVert q_{c,\alpha} c g e_{\varpi_\alpha }\rVert.
\]
Recall that
\[
	L = \bigcap_{\alpha \in \Pi \setminus \theta }\mathrm{Stab}_G e_{\varpi_\alpha}.
\]
Note that \(\mathbf{h}^{(c)}(g)\) only depends on \(g\) modulo \(L\) and therefore the function \(\mathbf{h}^{(c)}\) is well defined on the quotient \(G/L\).
Let us record some useful observations:
\begin{itemize}
	\item \( \Gamma^c\cap L \) is a lattice in \( L \); the Haar probability measure on \( L/(\Gamma^c\cap L) \) will be denoted \( \mathrm{m}_{L/(\Gamma^c\cap L)} \).
	\item \(\Gamma^c\cap L\) has finite index in \(\Gamma^c \cap P\); we let \( N_c = [\Gamma^c \cap P:\Gamma^c \cap L] \).
	\item We normalize the Haar measures on \( G \) and \( K \) so that \( G/\Gamma^c \) and \( K \) have volume \( 1 \).
		By~\cite[Theorem~8.32, Proposition~8.43]{knapp_lgbi}, one may then normalize the Haar measure on \( L \) so that the following formula holds for any continuous function \( f \) with compact support on \( G \):
		\[
			\int_G f(g)\,dg = \int_K \int_{\mathfrak{a}_\theta } \int_L f(k e^y l) e^{ \langle \varrho_X, y \rangle }\,dk\,dy\,dl.
		\]
	\item We may normalize the Haar measure on \( G/L \) so that for any continuous function \( \varphi \) with compact support on \( G/L \),
		\[
			\int_{G/L} \varphi = \int_K \int_{\mathfrak{a}_\theta } \varphi(ke^yL)e^{ \langle \varrho_X, y \rangle }\, dk\,dy. 
		\]
\end{itemize}
Write
\begin{align*}
	\#\{ v \in \Gamma c P\ |\ \mathbf{h}(v) \in \mathcal{D}_t\}
	& = \sum_{\gamma^c \in \Gamma^c/(\Gamma^c\cap P)} \mathbbm{1}_{\{\mathbf{h}(c \gamma^c P) \in \mathcal{D}_t\}} \\
	& = \frac{1}{N_c}\sum_{\gamma^c \in \Gamma^c/(\Gamma^c\cap L)} \mathbbm{1}_{\{\mathbf{h}^{(c)}(\gamma^c) \in \mathcal{D}_t\}}.
\end{align*}
Then, for \(t>0\), consider the function \(\Phi\) on \(G/\Gamma^c\) defined by
\[
	\Phi_t(g \Gamma^c)
	= \sum_{\gamma^c \in \Gamma^c/(\Gamma^c\cap L)} \mathbbm{1}_{\{\mathbf{h}^{(c)}(g \gamma^c) \in \mathcal{D}_t\}}.
\]
Our goal is to evaluate \(\Phi_t\) at the identity coset \(\Gamma^c \), as \[ \Phi_t(\Gamma^c) = \#\{ v \in \Gamma c P\ |\ \mathbf{h}(v) \in \mathcal{D}_t\}.\]
The inner product of \(\Phi_t\) with a smooth compactly supported function \(\psi\) on \(G/\Gamma^c\) can be rewritten
\begin{align*}
	\langle \psi, \Phi_t \rangle
	& = \int_{G/\Gamma^c} \sum_{\gamma^c \in \Gamma^c/(\Gamma^c\cap L)} \psi(g \Gamma^c)\,\mathbbm{1}_{\{\mathbf{h}^{(c)}(g \gamma^c) \in \mathcal{D}_t\}} \,d\mathrm{m}_{G/ \Gamma^c}(g \Gamma^c)\\
	& = \int_{G/(\Gamma^c\cap L)} \psi(g \Gamma^c) \mathbbm{1}_{\{\mathbf{h}^{(c)}(g(\Gamma^c\cap L)) \in \mathcal{D}_t\}} \,d\mathrm{m}_{G/(\Gamma^c\cap L)}(g(\Gamma^c\cap L)).
\end{align*}
By uniqueness of the Haar measure on \(G/(\Gamma^c\cap L)\), there exists a constant \(\kappa_{c}>0\) such that
\begin{align*}
	\langle \psi, \Phi_t \rangle 
	& = \kappa_{c} \int_{G/L} \mathbbm{1}_{\{\mathbf{h}^{(c)}(gL) \in \mathcal{D}_t\}} \left(\int_{L/(\Gamma^c \cap L)} \psi(g l \Gamma^c)\,d\mathrm{m}_{L/(\Gamma^c\cap L)}(l)\right)\,d\mathrm{m}_{G/L}(gL)\\
	& = \kappa_c \int_K \int_{\mathfrak{a}_\theta } \mathbbm{1}_{\{\mathbf{h}^{(c)}(ke^y) \in \mathcal{D}_t\}} \left(\int_{L/(\Gamma^c\cap L)} \psi(ke^yl \Gamma^c)\, d\mathrm{m}_{L/(\Gamma^c\cap L)}(l)\right) e^{ \langle \varrho_X, y \rangle } \, dy\, dk.
\end{align*}
Since the Euclidean norms are $K$-invarient, recalling that \( \mathcal{D}_t = tu + \mathcal{D}_0 \), the condition \( \mathbf{h}^{(c)}(ke^y) \in \mathcal{D}_t \) can be rewritten
\[
	y \in -q_c + t u + \mathcal{D}_0,
\]
where \( q_c \) is the unique vector in \( \mathfrak{a}_\theta  \) such that  \[
 \forall \alpha\in\Pi\setminus\theta,\quad   \langle \varpi_\alpha,q_c\rangle = \log \lVert q_{c,\alpha }ce_{\varpi_\alpha }\rVert .\]
Our assumption that \( u \in \mathring{C}_{\mathrm{eff}}^\vee \) is equivalent to \( \left\lfloor u \right\rfloor >0 \) and therefore \( \left\lfloor y \right\rfloor\) tends to \( +\infty \) uniformly for all \( y \) in \( \mathcal{D}_t \), as \( t \) tends to \( +\infty \).
By Theorem~\ref{th:equi}, this implies that the inner integral converges uniformly to \( \int_{G/\Gamma^c} \psi d\mathrm{m}_{G/\Gamma^c}\), so that
\begin{align*}
	\langle \psi, \Phi_t \rangle
	& \sim_{t\to+\infty} \kappa_c \cdot \nu(-q_c + tu + \mathcal{D}_0) \cdot \left( \int_{G/\Gamma^c} \psi d\mathrm{m}_{G/\Gamma^c}\right) \\
	& = \kappa_c \cdot e^{ \langle \varrho_X, tu - q_c \rangle } \cdot \nu(\mathcal{D}_0) \cdot \left( \int_{G/\Gamma^c} \psi d\mathrm{m}_{G/\Gamma^c}\right) .
\end{align*}
Taking \(\psi = \psi_\varepsilon \) to be a smooth function supported on \( \mathbf{B}_{G/\Gamma^c}(\Gamma^c,\varepsilon) \) satisfying \( \int_{G/\Gamma^c}\psi_\varepsilon d\mathrm{m}_{G/\Gamma^c}=1 \), and letting \( \varepsilon \) tend to zero, we find
\[
\Phi_t(\Gamma^c) \sim_{t\to+\infty} \kappa_c \cdot e^{ \langle \varrho_X, tu - q_c \rangle } \cdot \nu(\mathcal{D}_0).
\]
We may then sum over \( c \) in \( C \) to get
\[
		\# X(\mathbb{Q})_{\mathbf{h} \in \mathcal{D}_t}
		\sim_{t\to+\infty}
		\kappa_X \cdot \nu(\mathcal{D}_0) \cdot e^{ t \langle \varrho_X, u\rangle },
	\]
	with \begin{equation}\label{eq:kappa}
		\kappa_X = \sum_{c \in C} e^{ - \langle \varrho_X, q_c \rangle } \frac{\kappa_c}{N_c}. \qedhere
	\end{equation}
\end{proof}

\begin{remark}
	The proof given above can be readily adapted to also give a counting statement for rational points of large height in a fixed open subset of \( X(\mathbb{R}) \).
	Let \( \mu_X \) denote the unique \( K \)-invariant probability measure on \( X(\mathbb{R}) \) absolutely continuous with respect to the Lebesgue measure.
	Then, for every open subset \( \mathcal{O} \subset X(\mathbb{R})\) whose boundary satisfies \( \mu_X(\partial \mathcal{O})=0 \), one has
	\[
		\#\left\{ v \in \mathcal{O} \cap X(\mathbb{Q})\ |\ \mathbf{h}(v) \in \mathcal{D}_t \right\}
		\sim_{t\to+\infty}
		\kappa_X \cdot \mu_X(\mathcal{O}) \cdot \nu(\mathcal{D}_0) \cdot e^{ t \langle \varrho_X, u\rangle }.
	\]
\end{remark}

\section{Asymptotic behavior of zooming measures}
\label{sec:zoom}

The proof of Theorem~\ref{th:maini} from the introduction is very similar to the one given above for counting rational points with large height in \( X(\mathbb{Q}) \).
But in order to make the similarities more apparent, it is natural to modify slightly the definition of zooming measures.
This will also allow us to study general flag varieties, without any assumption on the \( \mathbb{Q} \)-rank.

\subsection{Rescaling on the tangent space}
\label{ss:rescale}

Fix a maximal compact subgroup \( K \) in \( G \), and for each $x\in X(\mathbb{R})$, choose an element \( s_x \) in \( K \) such that
\[
	x = s_xP.
\]
The tangent space to \( X = G/P \) at the base point \( x_0 = P \) is naturally identified with the quotient Lie algebra \( \mathfrak{g}/ \mathfrak{p} \simeq \mathfrak{u}^- \), and using the action of \( s_x \) on \( X \), one can further identify
\[
	T_xX \simeq T_{x_0}X \simeq \mathfrak{u}^-.
\]
Consider the big Schubert cell
\(
	U_x = \{ s_xe^ZP \mid Z \in \mathfrak{u}^- \}.
\)
It is an open neighborhood of \( x \) in \( X( \mathbb{R}) \).
To study the local distribution of rational points on \( X \), we shall use the projection
\[
	\begin{array}{lrcl}
		p_x\colon & U_x		 &\longrightarrow	& \mathfrak{u}^-\simeq T_x X
	\end{array}
\]
defined as the inverse of the diffeomorphism
\[
	\begin{array}{rcc}
		\mathfrak{u}^-	 &\longrightarrow	& U_x \\
		Z 		& \longmapsto & s_xe^ZP.
	\end{array}
\]
Note that \( p_x \) satisfies \( p_x(x)=0 \) and \( T_xp_x = \mathrm{Id}_{T_xX} \).
The rescaling on \( T_xX \simeq \mathfrak{u}^- \) is given by a one-parameter semigroup of dilations \( (a_t)_{t \in \mathbb{R}} \) that we now define.
For that, we use the direct sum decomposition (cf. \cite[\S2.2]{saxce_hdr})
\[
\mathfrak{u}^- = \bigoplus_{k\geq 1} m_k,
\]
where $m_k$ is the sum of all roots subspaces $\mathfrak{g}_\alpha$ such that $-\alpha$ is a positive root containing exactly $k$ elements not in $\theta$ in its decomposition into simple roots with multiplicity.
Then, for any element \( Z \) written as \(Z = \sum_{k\geq 1} z_k\) according to this decomposition, and for $t\in\mathbb{R}$, we let
\begin{equation}\label{eq:rescaling}
		a_t \cdot Z = \sum_{k\geq 1} e^{kt} z_k.
\end{equation}

\begin{remark}
	If \( \mathfrak{u}^- = m_1 \), then \( a_t \) simply acts by scalar multiplication by \( e^t \) on \( T_xX \).
	This happens if and only if the unipotent radical of \( P \) is abelian, which is the setting of Theorem~\ref{th:maini} in the introduction.
\end{remark}

\subsection{Zooming measures}
\label{ss:zm}

For our new definition of zooming measures, we not only use the rescaling map given in the above paragraph, but also control the multiheight using a family of compact subsets \( (\mathcal{D}_t) \), as in Section~\ref{sec:hc}.

\begin{definition}[Zooming measures -- general case]
Given a family of compact subsets \(\mathcal{D}_t \) in \( C_{\mathrm{eff}}^\vee \), the associated zooming measure with zoom factor $\tau>0$ at time \( t \), centered at \( x \), is the measure on \(T_xX\simeq \mathfrak{u}^-\) defined as
	\[
		\mu_{x, \tau, t} = \sum_{\substack{v\in X(\mathbb{Q})\cap U_x\\ \mathbf{h}(v) \in \mathcal{D}_{t}}} \delta_{a_{\tau t} \cdot p_x(v)}.
	\]
\end{definition}

The goal of this article is to find conditions on the family of subsets \( (\mathcal{D}_t) \) and the zoom factor \( \tau  \) under which one can describe the asymptotic behavior of the zooming measure centered at a generic point.
From the counting estimate  Theorem~\ref{th:count} and the fact that the dilation \( a_t \) scales the volume on the tangent space by a factor \( e^{-\tau t \langle \varrho_X,Y \rangle } \), where \( Y \in \mathfrak{a}_\theta \) is the unique element such that
\begin{equation}\label{eq:Y}
	\forall \alpha \in \Pi\setminus \theta, \quad \langle \alpha, Y \rangle = 1,
\end{equation}
it is natural to expect that for some \( \alpha_X>0 \), for small enough \( \tau>0 \), for all \( f \in \mathcal{C}_c(\mathfrak{u}^-) \),
\[
\mu_{x,\tau,t}(f) \sim_{t\to+\infty} \alpha_X \cdot e^{ -\tau t \langle \varrho_X, Y \rangle } \cdot \nu(\mathcal{D}_t) \cdot \left(\int_{ \mathfrak{u}^-} f\,d\mathrm{m}_{\mathfrak{u}^-}\right).
\]
We shall see in particular that such an estimate holds for generic points when the parabolic \( P \) is maximal and \( \mathcal{D}_t \) is an interval starting at \( 0 \) and of length proportional to \( t \).
This will allow us to derive Theorem~\ref{th:maini}.

\subsection{The zooming flow}

We now interpret rescaling \eqref{eq:rescaling} on the tangent space of \(X\) as the action of a well-chosen one-parameter multiplicative subgroup \((a_t)_{t \in \mathbb{R}}\) in \(G\). 
The subgroup \((a_t)_{t \in \mathbb{R}}\) is chosen so that \(\mathrm{Ad}\,a_t\) acts on \(\mathfrak{u}^-\) exactly as the semigroup of dilations introduced in paragraph~\ref{ss:rescale}.
More explicitly, we set $$a_t = e^{-tY},$$ where the element \( Y \in \mathfrak{a}_\theta \) is defined by \eqref{eq:Y}
We shall relate the asymptotic behavior of zooming measures to the dynamics of the action of \((a_t)_{t \in \mathbb{R}}\) on the finite-volume homogeneous space \(G/\Gamma \).
\begin{remark}
	This interpretation of the zooming flow as the adjoint action of some one-parameter subgroup is central in our approach to local distribution of rational points.
	Such an interpretation is only possible if the manifold \(X(\mathbb{R})\) is endowed with its Carnot-Carathéodory metric,
	which is not Riemannian in general, and this is the reason why we chose to place ourselves in that setting.
\end{remark}

In order to have a good understanding of the local distribution of rational points near a given point \(x \in X(\mathbb{R})\), we shall need some control on the orbit \((a_t s_x \Gamma)_{t > 0} \) in the space \(G/\Gamma \).
Precisely, the condition we shall request is that \((a_t s_x \Gamma)_{t>0}\) does not escape at positive speed in \(G/\Gamma \).
To make this statement rigorous, let us briefly recall some results from reduction theory for the space \(G/\Gamma \).

Let \(P_0 \subset P\) be a minimal \(\mathbb{Q}\)-parabolic in \(G\), \(U_0\) its unipotent radical, \(T \subset P_0\) a maximal split \(\mathbb{Q}\)-torus, \(A=T(\mathbb{R})^\circ\) the connected component of \(T(\mathbb{R})\), and \(M\) a maximal anisotropic \(\mathbb{Q}\)-subgroup of the centralizer \(Z(T)^\circ\).
Write $\mathfrak{a}=\mathrm{Lie}(A)$ and $\Pi\subset\mathfrak{a}^*$ for a basis of the root system of \((G,T)\) for an order associated to \(P_0\).
The \emph{positive Weyl chamber} \(\mathfrak{a}^+\subset\mathfrak{a}\) is the convex polytope defined by
\begin{equation}\label{eq:a+}
	\mathfrak{a}^+ = \{z \in \mathfrak{a}\ |\ \forall \alpha \in \Pi,\ \langle \alpha, z \rangle \geq 0 \}.
\end{equation}
Finally, fix a maximal compact subgroup \(K\) in \(G\), and let \(C_0 \subset G(\mathbb{Q})\) denote a finite set of representatives for \(\Gamma \backslash G(\mathbb{Q})/P_0\).
The fundamental results of Borel and Harish-Chandra's reduction theory \cite[Théorème~15.5 and Proposition~15.6]{borel_iga} show that there exist a compact neighborhood \(\omega \) of the identity in \(MU_0\) and a constant \(R\geq 0\) such that any element in \(G\) admits a Siegel decomposition
\[
g = k e^{-s} n c \gamma,
\]
with \(k \in K\), \(n \in \omega \), \(c \in C_0\) and \(\gamma \in \Gamma \)
and \(s \in \mathfrak{a}\) satisfying \(\operatorname{d}(s,\mathfrak{a}^+) \leq R\).
Such a decomposition is not unique in general, but the element \(s\) is uniquely determined up to a bounded element in \(\mathfrak{a}\); we shall denote it by \(s(g)\).

\begin{definition}[Zero rate of escape]
	An orbit \((a_t g \Gamma)_{t>0}\) is said to have \emph{zero rate of escape} in the space \(G/\Gamma \) if
	\(
	\lim_{t\to+\infty} \frac{1}{t}s(a_tg) = 0. 
	\)
\end{definition}

\begin{remark}
	This definition does not depend on the choices of the compact group \(K\) and the compact subset \(\omega \subset MU_0\) made for the construction of the Siegel decomposition.
\end{remark}

For our study of zooming measures, we shall use that the condition of zero rate of escape is generically satisfied, both for random elements chosen according to the Lebesgue measure, and for algebraic points outside rational constaints.
We summarize these facts in the proposition below, extracted from~\cite{saxce_hdr}.

\begin{proposition}[Generic zero rate of escape]
	\label{pr:zerorate}
	For \(x \in X(\mathbb{R})\), let \(s_x \in G\) be such that \(x = s_xP\).
	\begin{enumerate}
		\item For almost every \(x \in X(\mathbb{R})\) in the Lebesgue measure, the orbit \((a_ts_x \Gamma)_{t>0}\) has zero rate of escape in \(G/ \Gamma \).
		\item For all \(x \in X(\overline{\mathbb{Q}})\) not contained in any proper rational Schubert subvariety, the orbit \((a_t s_x \Gamma)_{t>0}\) has zero rate of escape in \(G/ \Gamma \).
	\end{enumerate}
\end{proposition}
\begin{proof}
	One should first note that the conclusion of the proposition is independent of the choice of the element \(s_x\) in \(G\) such that \(x = s_xP\).
	It is therefore enough to prove that for almost every \(g\) in \(G\) (for the first part), or every \(g\) in \(G(\overline{\mathbb{Q}})\) outside any proper rational Bruhat subvariety of \(G\) (for the second part), the orbit \((a_tg \Gamma)_{t>0}\) has zero rate of escape.
	The first part is an easy application of the Borel-Cantelli lemma, together with the fact that the Haar measure \(\mathrm{m}_{G/ \Gamma }\) on \(G/ \Gamma \) is preserved by \(a_t\) and satisfies, for some \(\tau>0\),
	\[
\mathrm{m}_{G/\Gamma }(\{g \Gamma \ |\ \lVert s(g)\rVert \geq \varepsilon t \} 
\lesssim e^{-\varepsilon \tau t},
\]
which can be seen from reduction theory and the construction of fundamental domains for \(G/\Gamma \), see \cite[Proposition~3.1.1]{saxce_hdr} for instance.
The second is a consequence of a parametric version of Schmidt's subspace theorem, as explained in \cite[Corollaire~5.2.3]{saxce_hdr}.
\end{proof}

\subsection{Convergence of zooming measures}
Our argument is similar to the one used in the proof of Theorem~\ref{th:count}.
To begin with, we decompose the zooming measure as a finite sum
\[
	\mu_{x,\tau,t} = \sum_{c \in C} \mu_{x,\tau,t}^{(c)},
\]
where each \(\mu_{x,\tau,t}^{(c)}\) is defined by
\[
	\mu_{x, \tau, t}^{(c)} = \sum_{\substack{v\in \Gamma c P \cap U_x\\ \mathbf{h}(v) \in \mathcal{D}_{t}}} \delta_{a_{\tau t} \cdot p_x(v)}.
\]
It suffices to study the convergence of each measure \(\mu_{x,\tau,t}^{(c)}\).
So we fix some element \(c\) in the finite set \(C\) of representatives of \(\Gamma \backslash G(\mathbb{Q}) / P\) and study the asymptotics of the zooming measure \(\mu_{x,\tau,t}^{(c)}\). 
We shall be able to prove equidistribution of zooming measures if the family of compact subsets \( \mathcal{D}_t \) in \( C_{\mathrm{eff}} \) satisfies a certain counting estimate for points in lattices translated by elements of relatively small norm, which we shall call \emph{low lattices}.
To state this property, define, for \( s \in G/L \),
\[
\Lambda(s) = (\log \lVert s e_{\varpi_\alpha} \rVert)_{\alpha \in \Pi \setminus \theta},
\]
and given a bounded open subset \( \mathcal{O} \) with smooth boundary in \( \mathfrak{u}^- \) and a vector \( d_0 \) in \( \mathfrak{a}_\theta \), let
\[
\mathcal{R}_t = \left\{ s \in G/L \ \left|\ 
	\begin{array}{l}
		s \in e^{\mathcal{O}}P \\
		\Lambda(a_{\tau t}^{-1}s) \in d_0 + \mathcal{D}_t
	\end{array}\right. \right\}.
\]
The technical counting statement we shall need in order to derive equidistribution of zooming measures is the following.

\begin{definition}[Effective counting in low lattices]
	\label{def:countvol}
	We say that the family of subsets \( \mathcal{D}_t \) satisfies \emph{effective counting in low lattices} for the zoom factor \( \tau  \) if for every bounded open subset \( \mathcal{O} \) with smooth boundary in \( \mathfrak{u}^- \) and any \( d_0 \in \mathfrak{a}_\theta \), there exist constants \(\alpha_{0,d_0}\) and $\eta>0$ such that for all \(t>0\) and all $g_2$ in $G$ such that \( \lVert g_2\rVert \leq e^{\eta t}\),
\[
	\sum_{\gamma_1 \in g_2\Gamma^c/\Gamma^c\cap L} \mathbbm{1}_{\mathcal{R}_t}(\gamma_1L)
	= \alpha_{0,d_0} \cdot \mathrm{m}_{G/L}(\mathcal{R}_t)\left(1 + O(e^{-\eta t})\right).
\]
\end{definition}

\begin{remark}
As we shall see below, it is not difficult to check that the volume \( m_{G/L}(\mathcal{R}_t) \) is comparable to \( e^{-\tau t \langle \varrho_X,Y \rangle }\nu(\mathcal{D}_t) \) when \( t \) is large.
A necessary condition to have effective counting in low lattices is therefore that the sets \( \mathcal{D}_t \) have \( \nu \)-measure tending to infinity faster than \( e^{\tau t \langle \varrho_X, Y \rangle } \).
\end{remark}

Our goal is to establish the following proposition.

\begin{proposition}
	\label{pr:eqc}
	Assume the family of subsets \( \mathcal{D}_t \) satisfies effective counting in low lattices for the zoom factor \( \tau>0 \).
	Assume in addition that for \( \varepsilon>0 \), the \( e^{-\varepsilon t} \)-neighborhood of \( \mathcal{D}_t \) satisfies \[ \nu(\mathcal{D}_t+B_{\mathfrak{a}_\theta}(0,e^{-\varepsilon t})) \sim_{t\to+\infty} \nu(\mathcal{D}_t) .\]
	Then, there exists a constant \(\alpha_X\) depending only on \(X\) such that the following holds.\\
	Let \(x\in X(\mathbb{R})\) be such that \((a_t s_x \Gamma)_{t>0} \) has zero rate of escape in \(G / \Gamma \).
	Then,  for every function \( f\in \mathcal{C}_c(\mathfrak{u}^-) \),
\[
	\mu_{x,\tau,t}(f) \sim_{t\to+\infty} \alpha_X \cdot e^{- \tau t \langle \varrho_X, Y\rangle }\cdot \nu(\mathcal{D}_t) \cdot\left(\int_{\mathfrak{u}^-} f\,d\mathrm{m}_{\mathfrak{u}^-}\right).
\]
\end{proposition}
\begin{remark}
	The condition \( \nu(\mathcal{D}_t+B_{\mathfrak{a}_\theta}(0,e^{-\varepsilon t})) \sim_{t\to+\infty} \nu(\mathcal{D}_t) \) should be interpreted as saying that the boundary of \( \mathcal{D}_t \) is regular enough.
	In all the examples that we shall consider, such as the ones suggested by Peyre and already studied in Section~\ref{sec:hc}, this condition is satisfied.
\end{remark}
\begin{proof}[Proof of Proposition~\ref{pr:eqc}]
It is enough to check that for each \( c \in C \), there exists a constant \( \alpha_c \) such that for any bounded open subset $\mathcal{O}\subset \mathfrak{u}^-$ with smooth boundary, one has the asymptotic equivalent
\[
	\mu_{x, \tau, t}^{(c)}(\mathcal{O}) \sim_{t \to +\infty} \alpha_c \cdot e^{ -\tau t \langle \varrho_X, Y\rangle }\cdot \nu(\mathcal{D}_t) \cdot \mathrm{m}_{\mathfrak{u}^-}(\mathcal{O}).
\]
Note that \(\mu_{x, \tau, t}^{(c)}(\mathcal{O})\) is equal to the number of rational points \(v \in \Gamma c P\) such that
\begin{equation}\label{eq:vcond}
	\left\{
		\begin{array}{l}
			\mathbf{h}(v) \in \mathcal{D}_t\\
			a_{\tau t} \cdot p_x(v) \in \mathcal{O}.
		\end{array}
	\right.
\end{equation}
Writing \(v= \gamma c P = c \gamma^c P\) in the decomposition \ref{eq:union}, observe that the elements \(v\) are in one-to-one correspondence with cosets \(\gamma^c (P\cap \Gamma^c)\) in \(\Gamma^c/(P\cap \Gamma^c)\).
To make explicit computations, it will be convenient to identify the vector space \(\mathfrak{a}_\theta \) with \(\mathbb{R}^{ \lvert \Pi \setminus \theta\rvert}\) through the map \( Z \mapsto ( \langle \varpi_\alpha, Z \rangle )_{\alpha \in \Pi \setminus \theta } \).
With this identification, and recalling also that \(p_x^{-1}\colon Z \mapsto s_xe^ZP\), conditions \eqref{eq:vcond} can be rewritten in terms of \(\gamma^c\) as
\begin{equation}
	\left\{
		\begin{array}{l}
			(\log \lVert q_{c,\alpha }c \gamma^c e_{\varpi_\alpha }\rVert)_\alpha \in \mathcal{ D}_t \\
			\gamma^c \in c^{-1}s_x a^{-1}_{\tau t}e^\mathcal{O} P.
		\end{array}
	\right.
\end{equation}
We now let
\[
	g_t = a_{\tau t} s_x^{-1} c
\]
and use the change of variables
\[
	\gamma_1 = g_t \gamma^c \in g_t \Gamma^c. 
\]
So we want to count the number of cosets \(\gamma_1\) in \(g_t \Gamma^c/ (P\cap \Gamma^c)\) satisfying
(note that \(a^{-1}_{\tau t} \gamma_1 = s_x^{-1}c \gamma^c\) and that \(s_x\) preserves the norm in each representation \( V_\alpha \))
\begin{equation*}
	\left\{
		\begin{array}{l}
			(\log \lVert q_{c,\alpha }a_{\tau t}^{-1} \gamma_1 e_{\varpi_\alpha }\rVert)_\alpha \in \mathcal{ D}_t \\
			\gamma_1 \in e^\mathcal{O} P.
		\end{array}
	\right.
\end{equation*}
These conditions depend on \( \gamma_1 \) only modulo the subgroup \[ L = \bigcap_{\alpha \in \Pi \setminus \theta }\mathrm{Stab}_G e_{\varpi_\alpha}.\]
Let
\begin{equation}\label{eq:d0}
		d_0 = (-\log \lvert q_{c,\alpha }\rvert)_{\alpha \in \Pi \setminus \theta}.
\end{equation}
Denoting as before \( N_c = [\Gamma^c\cap P : \Gamma^c \cap L] \), one has

\begin{align*}
	\mu_{x,\tau,t}^{(c)}(\mathcal{O}) & = \sum_{\gamma_1 \in g_t\Gamma^c/\Gamma^c\cap P} \mathbbm{1}_{\{\gamma_1 \in e^\mathcal{O} P\ \mathrm{and}\ \Lambda(a_{\tau t}^{-1}\gamma_1) \in d_0 + \mathcal{D}_t\}} \\
				& = \frac{1}{N_c}\sum_{\gamma_1 \in g_t\Gamma^c/\Gamma^c\cap L} \mathbbm{1}_{\mathcal{R}_t}(\gamma_1L).
\end{align*}
By our assumption that the family \( (\mathcal{D}_t)_{t >0} \) satisfies effective counting in low lattices, there exists $\eta>0$ such that as \(t\) tends to \(+\infty\), for every $g_2$ such that $\norm{g_2}\leq e^{\eta t }$,
	\begin{equation}\label{eq:effectivecounting}
		\sum_{\gamma_1 \in g_2\Gamma^c/\Gamma^c\cap L} \mathbbm{1}_{\mathcal{R}_t}(\gamma_1L)
	= \alpha_{0,c} \cdot \mathrm{m}_{G/L}(\mathcal{R}_t) \cdot \left(1 + O(e^{-\eta t})\right).
	\end{equation}
Moreover, recalling that \(g_t = a_{\tau t} s_x^{-1} c\) and that \((a_t s_x^{-1} \Gamma)_{t>0} \) has zero rate of escape, we can always find, for \(t>0\) large enough, an element \(g_2\) (implicitly depending on \( t \)) such that \(\lVert g_2 \rVert \leq e^{\eta t}\) and
\[
	g_t \Gamma^c = g_2 \Gamma^c.
\]
Thus, to conclude, all we need is to estimate the volume of the set \(\mathcal{R}_t\).
For that, we decompose an element \(gL\) in \(\mathcal{R}_t\) as
\[
g = n e^y L,\ n\in U^-,\ y \in \mathfrak{a}_\theta.
\]
From the Haar measure decomposition~\cite[Theorem~8.32]{knapp_lgbi} applied first to $G=U^-P$ and then to $P=e^{\mathfrak{a}_\theta}L$, (note that the modulus for the Haar measure on $P$ satisfies \(\Delta_P(e^y)=e^{ \langle \varrho_X, y \rangle}\)), we obtain, with the appropriate normalization,
\begin{align*}
\mathrm{m}_{G/L}(\mathcal{R}_t) & = \int_{U^-} \int_{\mathfrak{a}_\theta } \mathbbm{1}_{\mathcal{R}_t}(ne^y) e^{ \langle \varrho_X, y \rangle }\, dn\, dy \\
				& = \int_{U^-} \mathbbm{1}_{\{n \in e^{\mathcal{O}}\}} \int_{\mathfrak{a}_\theta}  \mathbbm{1}_{\{\Lambda(a_{\tau t}^{-1}ne^y) \in d_0 + \mathcal{D}_t\}} e^{ \langle \varrho_X, y \rangle }\,d y\,dn.
\end{align*}
Rewriting
\[
	a_{\tau t}^{-1} \gamma_1 = a_{\tau t }^{-1} n e^{y} = (a_{\tau t}^{-1} n a_{\tau t})a_{\tau t}^{-1} e^y = (a_{\tau t}^{-1} n a_{\tau t})e^{\tau t Y+y}
\]
and using that \( a_{\tau t}^{-1} n a_{\tau t} \) is exponentially close to the identity, the condition \( \Lambda(a_{\tau t}^{-1}n e^y) \in d_0 + \mathcal{D}_t\) is, up to an exponentially small error, equivalent to
\[
y \in d_0 -  \tau t Y + \mathcal{D}_t.
\]
Therefore, we may conclude
%
\[
	\mathrm{m}_{G/L}(\mathcal{R}_t) \sim_{t\to+\infty} e^{ \langle \varrho_X, d_0 \rangle } e^{- \tau t \langle \varrho_X, Y \rangle } \mathrm{m}_{\mathfrak{u}^-}(\mathcal{O}) \nu(\mathcal{D}_t).
\]
This yields the desired result, with constant \(\alpha_c = \alpha_{0,c}e^{ \langle \varrho_X,d_0 \rangle }\).
\end{proof}

\subsection{Effective counting in low lattices}
\label{ss:effcount}
In this paragraph, we apply Proposition~\ref{pr:eqc} to derive Theorem~\ref{th:maini} from the introduction and another convergence result for zooming measures in general flag varieties, without restriction on the rank.
This amounts to checking that certain families of subsets \((\mathcal{D}_t) \) indeed satisfy the effective counting statement introduced in the previous paragraph.

\bigskip
We shall need an effective version of Theorem \ref{th:equi}, stated below.
It follows from  Shi~\cite[Theorem~1.6]{shi}, and was also observed independently by Dabbs, Kelly and Li~\cite[Theorem~2]{Dabbs-Kelly-Li} for the group \( \operatorname{SL}_d \), with a proof that can be adapted to cover the case of a general semisimple \( \mathbb{Q} \)-group without compact factors, as explained in~\cite[\S1.3]{Dabbs-Kelly-Li}.
In the statement below, we use a Sobolev norm on smooth functions on \( G/\Gamma  \), defined in the following way.
Fix a basis \( (u_i)_{1\leq i\leq d} \) for the Lie algebra \( \mathfrak{g} \) of \( G \) and for \( \boldsymbol{\alpha} = (\alpha_1,\dots,\alpha_d) \) in \( \mathbb{N}^d \), consider the associated left-invariant differential operators \( D_{\boldsymbol{\alpha}} =  \partial_{u_1}^{\alpha_1}\dots\partial_{u_d}^{\alpha_d} \).
For \( k \geq 1 \), the Sobolev norm \( \mathcal{S}_k(\varphi) \) of a smooth compactly supported function \( \varphi \) on \( G/\Gamma \) is defined by
\[
\mathcal{S}_k(\varphi) = \max_{ \boldsymbol{\alpha} :\ \alpha_1+\dots+\alpha_d \leq k} \lVert D_{\boldsymbol{\alpha}}\varphi \rVert_\infty.
\]
\begin{theorem}[Effective equidistribution of translated periodic orbits]
	\label{th:effequi}
	\hspace{-.5cm} Let \( G \) be a semisimple \( \mathbb{Q} \)-group without compact factors and \( \Gamma = G(\mathbb{Z}) \) be an arithmetic subgroup.
	Let \( P \) be a parabolic \( \mathbb{Q} \)-subgroup defined by a subset of simple roots \( \theta \subset \Pi \) and let
	\(
		L = \bigcap_{\alpha \in \Pi \setminus \theta }\mathrm{Stab}_G e_{\varpi_\alpha}.
	\)
	Finally, for \( y \in \mathfrak{a}_\theta \), set
	\(
		\left\lfloor y \right\rfloor = \min_{\alpha \in \Pi \setminus \theta} \varpi_\alpha(y). 
	\)

	Then, there exists a Sobolev norm \( \mathcal{S}_k \) on \( \mathcal{C}_c^\infty(G/\Gamma) \) and 
	constants \(C_0, \delta>0\) such that for all \( \varphi \in \mathcal{C}_c^\infty(G/\Gamma) \) and all \( y \in \mathfrak{a}_\theta \),
	\[
	\left\lvert \int_{L/(\Gamma\cap L)} \varphi(e^y l \Gamma)\, d\mathrm{m}_{L/(\Gamma\cap L)}(l) - \int_{G/ \Gamma} \varphi \, d\mathrm{m}_{G/\Gamma}\right\rvert
	\leq C_0 \cdot \mathcal{S}_k(\varphi)\cdot e^{-\delta \lfloor y \rfloor}.
	\]
\end{theorem}

We are now ready to prove Theorem~\ref{th:maini}, giving equidistribution of zooming measures on rank-one varieties.
For the reader's convenience, we recall the statement below, before its proof.
We note in passing that now that the zooming measures have been defined using the correct rescaling, the assumption that the unipotent radical of \( P \) is abelian is no longer necessary; it suffices to assume that \( P \) is a maximal parabolic subgroup. 

\begin{theorem}[Generic local distribution]
	\label{th:gld}
Let $G$ be a simple $\mathbb{Q}$-group, $P$ a maximal $\mathbb{Q}$-parabolic subgroup and write $X=G/P$ for the associated flag variety.
Assume \( X \) is endowed with the height \( H_\chi \) associated to an irreducible representation with highest weight \( \chi \), and denote by \( \beta_\chi=\frac{1}{\langle \chi, Y \rangle} \) the almost sure diophantine exponent.\\
Consider the zooming measure on the tangent space \( T_xX \) defined as
\[
	\mu_{x,\tau,t} = \sum_{\substack{v\in X(\mathbb{Q})\cap U_x\\ \mathrm{H}(v)\leq e^t}} \delta_{a_{\tau t} \cdot p_x(v)}.
\]
For all $\tau\in(0,\beta_\chi)$, for all $x$ in $X(\overline{\mathbb{Q}})$ not contained in any proper rational Schubert variety, for all $f\in \mathcal{C}_c(T_xX)$,
\[
	\mu_{x,\tau,t}(f)
	\sim_{t\to+\infty} \alpha_X\cdot \frac{ \langle \chi, Y \rangle }{ \langle \varrho_X, Y \rangle } \cdot e^{t \langle \varrho_X, Y \rangle (\beta_\chi - \tau) } \cdot \left( \int_{\mathfrak{u}^-}f\,d \mathrm{m_{\mathfrak{u}^-}} \right),
\]
where $\mathrm{m}$ is the Lebesgue measure on $\mathfrak{u}^-\simeq T_xX$ 
and \( \alpha_X \) is the constant from Proposition~\ref{pr:eqc}.
Moreover, the same estimate holds for Lebesgue almost every $x$ in $X(\mathbb{R})$.
\end{theorem}
\begin{proof}
	Our assumptions on \( x \) ensure that if \( s_x \) is any element in \( G \) such that \( x=Ps_x \), then the orbit \( (a_t s_x \Gamma)_{t>0} \) has zero rate of escape in \( G/\Gamma \).
	Since \( P \) is a maximal parabolic subgroup, there exists a unique root \( \alpha \) such that \( \Pi \setminus \theta = \left\{ \alpha \right\}  \).
The space \( \mathfrak{a}_\theta \) is one-dimensional and \( Y \in \mathfrak{a}_\theta \) is uniquely defined by the condition \(\langle\alpha,Y\rangle=1\).
	The logarithmic height \( \mathrm{h}_\chi \) on \( X=G/P \) is given by an irreducible representation of \( G \) with highest weight \( \chi = n \varpi_\alpha \), for some positive integer \( n \).
	The height \( \mathrm{h}_\chi \) is recovered from the multiheight by the formula \( \mathrm{h}_\chi(v) = \langle \chi, \mathbf{h}(v) \rangle  \).
	We consider the family of sets
	\[
	\mathcal{D}_t = \left\{ z \in \mathfrak{a}_\theta\ |\ 0 \leq \langle \chi, z \rangle \leq t \right\}.
\]
The condition \( \mathbf{h}(v) \in \mathcal{D}_t \) is equivalent to \( H_\chi(v) \leq e^t \).
Moreover, it is readily checked that for \( \varepsilon>0 \), the \( \nu \)-measure of the \( e^{-\varepsilon t} \)-neighborhood of \( \mathcal{D}_t \) is equivalent to that of \( \mathcal{D}_t \) as \( t \) tends to \( +\infty \).
Therefore, in order to apply Proposition~\ref{pr:eqc}, it suffices to show that the family \( \mathcal{D}_t \) satisfies the effective counting for low lattices for any zoom factor \( \tau < \beta_\chi  \).

\bigskip
For \( t>0 \), consider the function \( F_t \) on \( G/\Gamma^c \) given by the expression
\[
	F_t(g \Gamma^c) =  
	\sum_{\gamma \in \Gamma^c/\Gamma^c\cap L} \mathbbm{1}_{\{g \gamma L \in \mathcal{R}_t \}}.
\]
Our goal is to evaluate \(F_t\) at the element \(g_2 \Gamma^c\), which is relatively close to the identity coset \(\Gamma^c \).
As in the proof of Theorem~\ref{th:count}, we may rewrite the inner product of \(F_t\) with a smooth compactly supported function \(\psi\) on \(G/\Gamma^c\) as
\begin{align*}
	\langle \psi, F_t \rangle
	& = \int_{G/\Gamma^c} \psi(g \Gamma^c)\sum_{\gamma \in \Gamma^c/\Gamma^c\cap L} \mathbbm{1}_{\mathcal{R}_t}(g \gamma L)  \,d\mathrm{m}_{G/ \Gamma^c}(g \Gamma^c)\\
	& = \alpha_{0,c} \int_{G/L} \mathbbm{1}_{\mathcal{R}_t}(gL) \left(\int_{L/(\Gamma^c \cap L)} \psi(g l \Gamma^c)\,d\mathrm{m}_{L/(\Gamma^c\cap L)}(l)\right)\,d\mathrm{m}_{G/L}(g).
\end{align*}
For \(g\) in \(\mathcal{R}_t\), we may write \(g=n e^y L\), with \(n \in e^{\mathcal{O}}\) and \(y \in \mathfrak{a}_\theta \).
The element \( g \) belongs to \( \mathcal{R}_t \) if and only if \( \Lambda(a_{\tau t}^{-1}g) \in d_0 + \mathcal{D}_t \). As in the proof of Proposition \ref{pr:eqc}, this is, up to an exponentially small error, equivalent to
\[
	y \in d_0 - \tau t Y + \mathcal{D}_t
\]
i.e. 
\[
	\langle \chi, d_0 \rangle  - \frac{\tau }{\beta_\chi}t \leq \langle \chi, y \rangle \leq \langle \chi,d_0 \rangle + \left(1- \frac{\tau }{\beta_\chi}\right)t,
\] where $d_0$ is defined by \eqref{eq:d0}.
We split \( \mathcal{R}_t \) into two subsets
\[
	\mathcal{R}_t' = \left\{ g=ne^y \in \mathcal{R}_t\ |\ \langle \chi, y \rangle \geq \left(1-\frac{\tau }{\beta_\chi}\right)\frac{t}{2} \right\} 
	\quad\mbox{and}\quad
	\mathcal{R}_t'' = \mathcal{R}_t \setminus \mathcal{R}_t'.
\]
For \( g \) in \( \mathcal{R}_t' \), the lower bound
\[
	\langle \varpi_\alpha, y \rangle  \gtrsim \left(1-\frac{\tau }{\beta_\chi}\right) t
\]
allows us to apply Theorem~\ref{th:effequi}, to conclude that for some \(\eta_0=\eta_0(c,\tau)>0\), uniformly for all \(g\) in \(\mathcal{ R}_t'\),
\begin{equation}
	\label{eq:spsi}
	\left\lvert \int_{L/(\Gamma^c \cap L)} \psi(g l \Gamma^c)\,d\mathrm{m}_{L/(\Gamma^c\cap L)}(l) - \int_{G/\Gamma } \psi d \mathrm{m}_{G/L}\right\rvert
	\lesssim e^{-\eta_0t} \mathcal{S}(\psi).
\end{equation}
Now if \(\eta_1\) is large enough compared to \(\eta_2\), our assumption that \( \lVert g_2\rVert \leq e^{\eta_2 t}\) allows us to take for \(\psi\) a smooth approximate unit in a neighborhood of \(g_2 \Gamma^c \) satisfying:
\begin{enumerate}
\item $\int_{G/\Gamma^c}\psi = 1$;
\item $\mathrm{Supp}\psi \subset \mathbf{B}_{G/\Gamma^c}(g_2,e^{-\eta_1 t})$;
\item $\mathcal{S}(\psi) \leq e^{C \eta_1 t}$, where \(C\) is some constant depending on the degree of the Sobolev norm.
\end{enumerate}
Combining \eqref{eq:spsi} with the upper bound on $\mathcal{S}(\psi)$ and assuming we chose \(\eta_1\) so that \(\eta_0>(C+1)\eta_1\), we find
\begin{align*}
	\hspace{-2cm}	& \int_{G/L} \mathbbm{1}_{\mathcal{R}_t'}(gL) \left(\int_{L/(\Gamma^c \cap L)} \psi(g l \Gamma^c)\,d\mathrm{m}_{L/(\Gamma^c\cap L)}(l)\right)\,d\mathrm{m}_{G/L}(g)\\
	& \hspace{6cm} = \alpha_{0,c} \cdot \mathrm{m}_{G/L}(\mathcal{R}_t') \cdot \left(1+O(e^{-t\eta_1})\right).
\end{align*}
On the other hand, it is easy to check that for some \( \eta_0'>0 \) depending on \( \tau-\beta_\chi>0\), one can bound \( \mathrm{m}_{G/L}(\mathcal{R}_t'')\lesssim e^{-\eta_0' t} \mathrm{m}_{G/L}(\mathcal{R}_t)) \), so that in the end, one finds
\[
	\langle \psi, F_t \rangle
	= \alpha_{0,c} \cdot \mathrm{m}_{G/L}(\mathcal{R}_t) \cdot \left(1+O(e^{-t\eta_1})\right).
\]
We now explain how the assumption on the support of \(\psi\) can be used to show that \( \langle F_t, \psi\rangle \) is close to \( F_t(g_2 \Gamma^c)\).
The trick is to replace \(\mathcal{R}_t\) by a similarly defined but slightly smaller open set \( \mathcal{R}_t^- \)
such that, provided \(\eta_2>0\) is small enough compared to \(\eta_1\),
\begin{enumerate}
	\item \( \mathrm{m}_{G/L}(\mathcal{R}_t^-) \geq ( 1 - e^{-t \eta_2} ) \mathrm{m}_{G/L}(\mathcal{R}_t) \);
	\item for every \(h \in \mathbf{B}_G\left(1,e^{-\eta_1 t}\right)\), one has \(h \mathcal{R}_t^- \subset \mathcal{R}_t\).
\end{enumerate}
Then, we consider
\[
	F_t^-(g \Gamma^c) = \sum_{\gamma \in \Gamma^c/\Gamma^c\cap L} \mathbbm{1}_{\{g \gamma L \in \mathcal{R}^-_t \}}.
\]
For each \(g \Gamma^c\) in \(\mathrm{Supp}\psi \), we may write \(g \Gamma^c = h g_2 \Gamma^c \) for some \(h\) in \(\mathbf{B}_G(1,e^{-t\eta_1})\), so 
\[
	F_t^-(g \Gamma^c) \leq F_t(g_2 \Gamma^c)
\]
and therefore
\[
	\langle \psi, F_t^- \rangle \leq F_t(g_2 \Gamma^c). 
\]
Applying the above estimate to \(F_t^-\), we find
\begin{align*}
	F_t(g_2 \Gamma^c)
	& \geq \alpha_{0,c} \cdot \mathrm{m}_{G/L}(\mathcal{R}_t^-) \cdot \left( 1 - O(e^{-t\eta_1})\right)\\
	& \geq \alpha_{0,c} \cdot \mathrm{m}_{G/L}(\mathcal{R}_t) \cdot \left( 1 - O(e^{-t\eta_2})\right).
\end{align*}
A similar argument using a slightly larger set \(\mathcal{R}_t^+\) to control \(\mathcal{R}_t\) from above yields the analogous upper bound
\[
	F_t(g_2 \Gamma^c) \leq \alpha_{0,c} \cdot \mathrm{m}_{G/L}(\mathcal{R}_t) \cdot \left(1 + O(e^{-t\eta_2})\right).
\]
This shows that the family \( \mathcal{D}_t \) satisfies the effective counting for low lattices, and applying Proposition~\ref{pr:eqc}, one gets
\[
	\mu_{x,\tau,t}(f)
	\sim_{t\to+\infty} \alpha_X\cdot e^{- \tau t \langle \varrho_X, Y \rangle } \cdot \nu(\mathcal{D}_t) \cdot \left( \int_{\mathfrak{u}^-}f\,d \mathrm{m_{\mathfrak{u}^-}} \right).
\]
Noting that
\[
	\nu(\mathcal{D}_t) = \int \mathbbm{1}_{ 0 \leq \langle \chi,y \rangle \leq t} e^{ \langle \varrho_X, y \rangle }\,dy
	= \int_0^t e^{s\frac{ \langle \varrho_X, Y \rangle }{ \langle \chi, Y \rangle }}\,ds
	\sim_{t\to+\infty} \frac{ \langle \chi, Y \rangle }{ \langle \varrho_X, Y \rangle }e^{t\frac{ \langle \varrho_X, Y \rangle }{ \langle \chi, Y \rangle }}
\]
and recalling that \( \beta_\chi \cdot \langle \chi,Y \rangle =1 \), this yields
\[
	\mu_{x,\tau,t}(f)
	\sim_{t\to+\infty} \alpha_X\cdot \frac{ \langle \chi, Y \rangle }{ \langle \varrho_X, Y \rangle } \cdot e^{t \langle \varrho_X, Y \rangle (\beta_\chi - \tau) } \cdot \left( \int_{\mathfrak{u}^-}f\,d \mathrm{m_{\mathfrak{u}^-}} \right). \qedhere
\]
\end{proof}

One can also use Proposition~\ref{pr:eqc} to study the local distribution of rational points on a general flag variety, without any restriction on the rank.
This is the content of the result below, which allows a small zoom factor \( \tau >0\) for the local distribution of rational points with height in a moving compact set of the form \( \mathcal{D}_t = tu + \mathcal{D}_0 \), as in Section~\ref{sec:hc}.

\begin{theorem}[Local distribution for bounded domains in \( \mathring{C}_{\mathrm{eff}}^\vee \)]
	\label{th:gldg}
	Let \( G \) be a semisimple $\mathbb{Q}$-group, $P$ a maximal $\mathbb{Q}$-parabolic subgroup and write $X=G/P$ for the associated flag variety.
	Fix an element \( u \in \mathring{C}_{\mathrm{eff}}^\vee \), a compact domain \( \mathcal{D}_0 \) with smooth boundary in \( \mathfrak{a}_\theta \), and for \( t>0 \), let \( \mathcal{D}_t = tu + \mathcal{D}_0 \).\\
There exists \( \tau_0>0 \) such that for all $\tau\in(0,\tau_0)$, for all $x$ in $X(\overline{\mathbb{Q}})$ not contained in any proper rational Schubert variety, for all $f\in \mathcal{C}_c(T_xX)$,
\[
	\mu_{x,\tau,t}(f)
	\sim_{t\to+\infty} \alpha_X\cdot e^{-\tau t \langle \varrho_X, Y \rangle} \cdot \nu(\mathcal{D}_t) \cdot \left( \int_{\mathfrak{u}^-}f\,d \mathrm{m_{\mathfrak{u}^-}} \right).
\]
The same estimate holds for Lebesgue almost every $x$ in $X(\mathbb{R})$.
\end{theorem}
\begin{proof}
	The proof is very similar to the one of the previous theorem, so we only give a sketch to explain the main adjustments.
	All that needs checking is the effective counting statement for low lattices.
	First write
	\[
		\langle \psi, F_t \rangle
		= \alpha_{0,c} \int_{G/L} \mathbbm{1}_{\mathcal{R}_t}(gL) \left(\int_{L/(\Gamma^c \cap L)} \psi(g l \Gamma^c)\,d\mathrm{m}_{L/(\Gamma^c\cap L)}(l)\right)\,d\mathrm{m}_{G/L}(g).
	\]
	For \( g \in \mathcal{R}_t \), write \( g = k e^y \), with \( k \in K \) and \( y \in \mathfrak{a}_\theta \).
	From the assumption that \( u \in \mathring{C}_{\mathrm{eff}}^\vee \), one deduces that there exists \( \eta_0>0 \) such that for \( \tau>0 \) small enough, uniformly over all \( g \) in \( \mathcal{R}_t \), one has a lower bound \( \left\lfloor y \right\rfloor \geq \eta_0 t \).
	Applying Theorem~\ref{th:effequi}, this implies that for \( \eta_1>0 \) sufficiently small, if \( \psi \) satisfies \( \int_{G/\Gamma^c}\psi = 1 \) and has Sobolev norm \( \mathcal{S}(\psi) \leq e^{-\eta_1 t} \), then
\[
	\langle \psi, F_t \rangle
	= \alpha_{0,c} \cdot \mathrm{m}_{G/L}(\mathcal{R}_t) \cdot \left(1+O(e^{-t\eta_1})\right).
\]
If \( \eta_2>0 \) is small enough and \( g_2 \in G \) satisfies \( \lVert g_2 \rVert\leq e^{\eta_2 t} \), we may take \( \psi \) satisfying the above conditions and supported in the ball \( \mathbf{B}_{G/\Gamma^c}(g_2,e^{-\eta_1 t}) \).
At this point, it is important to note that the constants \( \eta_0 \), \( \eta_1 \) and \( \eta_2 \) can be chosen independently of \( \tau  \).
Then, if \( \tau  \) is sufficiently small compared to \( \eta_1 \), observe that for all \( h \) in \( \mathbf{B}_G(1,e^{-\eta_1 t}) \),
\[
	\Lambda(a_{\tau t}^{-1}hg) = \Lambda(a_{\tau t}^{-1} g) + O(e^{-\eta_1 t/2}).
\]
This allows us to construct sets \( \mathcal{R}_t^- \) and \( \mathcal{R}_t^+ \) as in the proof of Theorem~\ref{th:gld}, to show that \( \langle \psi, F_t \rangle \sim_{t\to+\infty} F_t(g_2 \Gamma^c) \) and therefore conclude the proof.
\end{proof}

\begin{remark}
	Adapting the argument from Theorem~\ref{th:gld}, one can derive the asymptotic equivalent under the much more restrictive condition that \( u- \tau Y\) lies in the interior of the positive Weyl chamber \(\mathfrak{a}_\theta^+ =\mathfrak{a}_\theta\cap \mathfrak{a}^+\), where we recall \eqref{eq:a+}.
	Note that this can only happen if \( u \) itself lies in \( \mathfrak{a}_\theta^+ \), so that the compact sets \( \mathcal{D}_t \) cannot move in any direction in the dual effective cone.
\end{remark}

\section{Concluding remarks and open problems}\label{se:openprob}

\subsection*{Flag varieties of \( \mathbb{Q} \)-rank more than one}
	It would be nice to generalize Theorem~\ref{th:maini} to the case of an arbitrary flag variety endowed with the anticanonical height.
	For that, one needs to consider the sets
	\[
		\mathcal{D}_t = \left\{ y \in \mathring{C}_{\mathrm{eff}}^\vee\ |\ \langle \varrho_X,u \rangle \leq t \right\}.
	\]
				\begin{figure}[h]
		\centering
		\begin{minipage}{0.47\textwidth}
			\centering
			\includegraphics[scale=0.4]{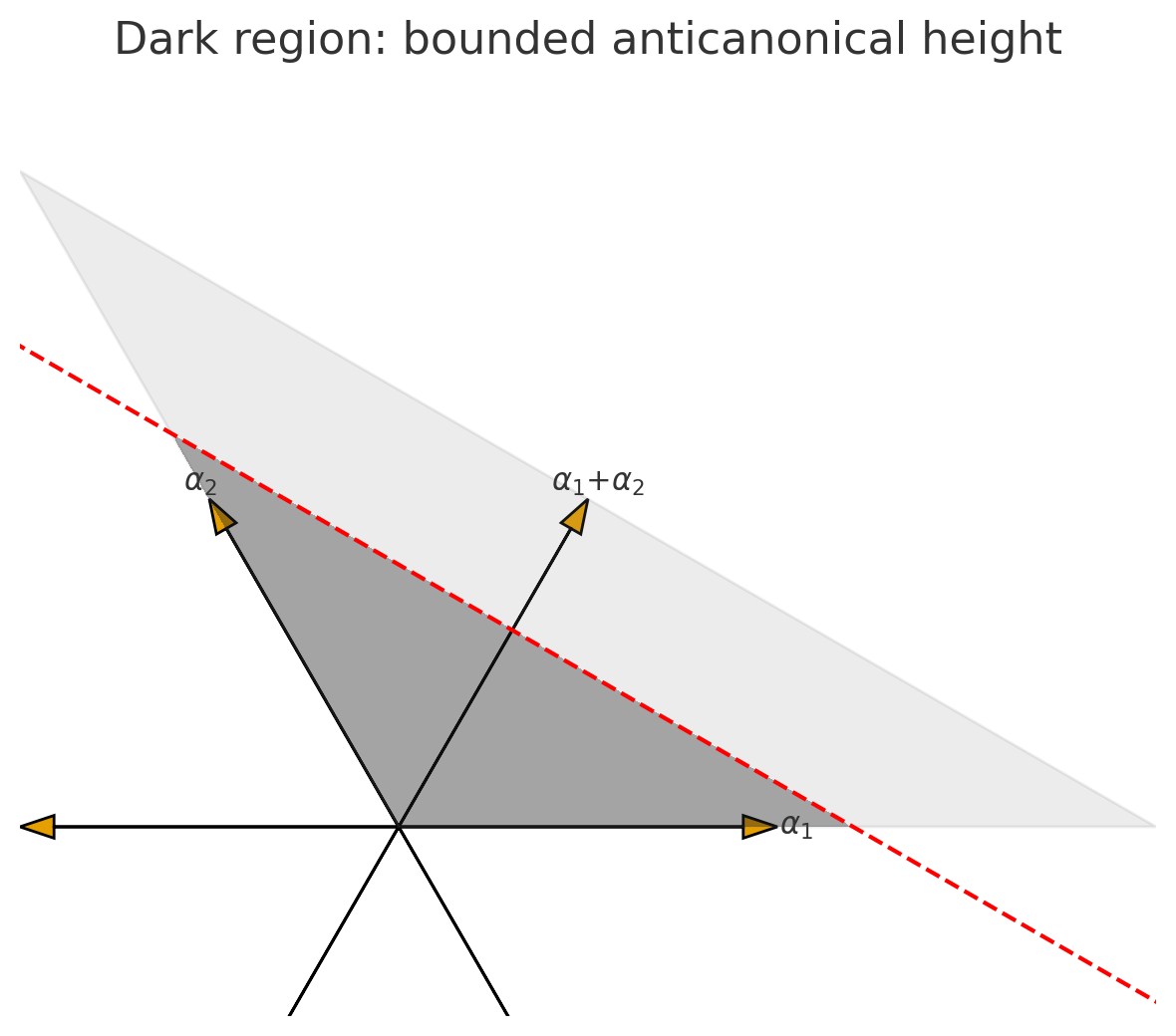}
		\end{minipage}
		\begin{minipage}{0.47\textwidth}
			\centering
			\includegraphics[scale=0.4]{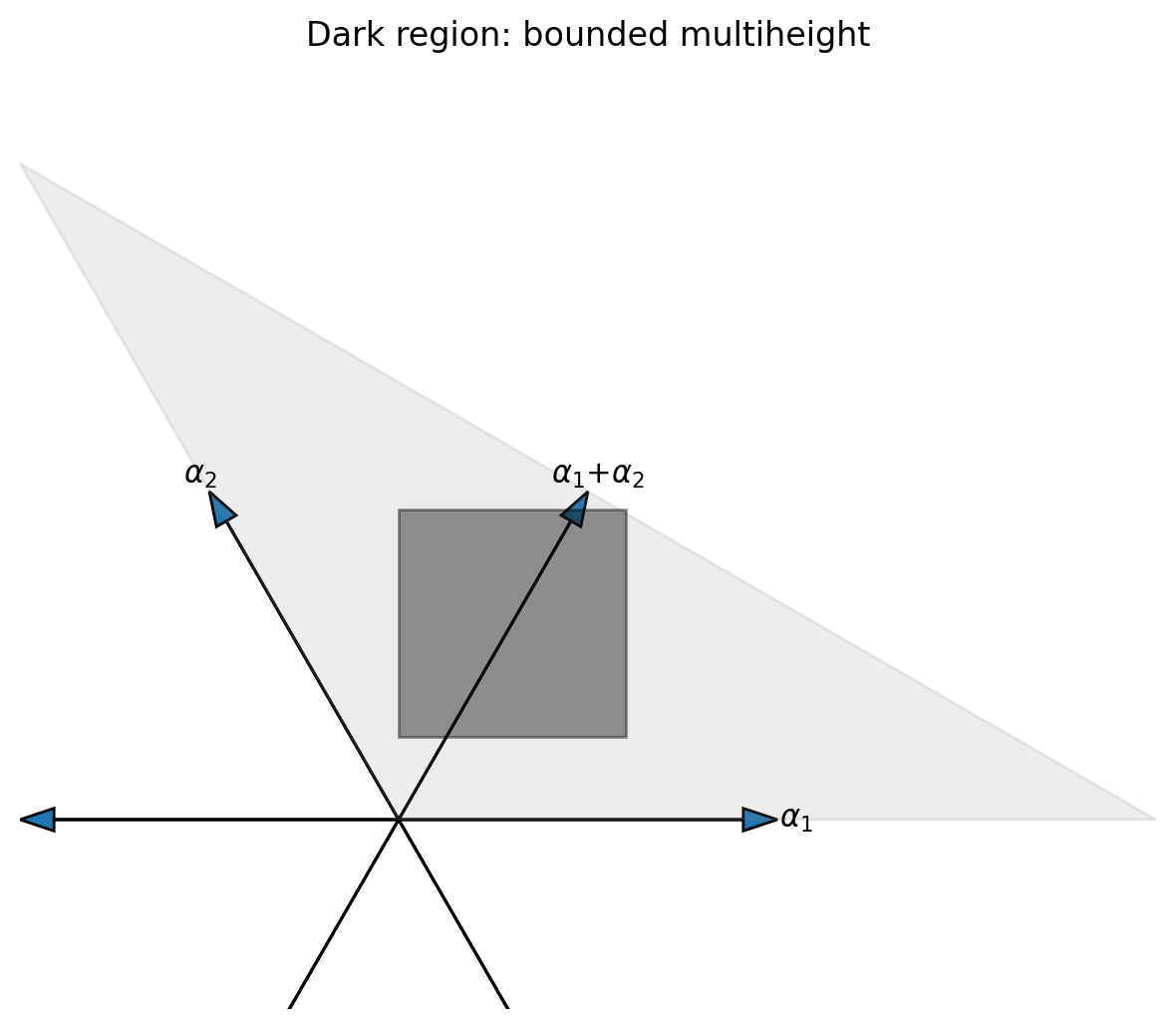}
		\end{minipage}
		\caption{\small The complete flag variety under $SL_3$}
	\end{figure}
	
	One issue is that for \( \varepsilon>0 \), the \( \nu \)-measure of the set of elements \( y \in \mathcal{D}_t \) satisfying \( \left\lfloor y \right\rfloor \geq \varepsilon t \) is not exponentially negligible compared to \( \nu(\mathcal{D}_t) \), unless \( \dim \mathfrak{a}_\theta = 1 \).
	This is the reason why it is more convenient to restrict the height to a moving compact subset of the form
	\[
		\mathcal{D}_t = tu + \mathcal{D}_0, \qquad\mbox{with}\ u \in \mathring{C}_{\mathrm{eff}}^\vee, 
\]
as in Theorem~\ref{th:gldg}. 
In this context, using notation from the proof of that theorem, it is not difficult to check that the weak counting statement
\[
		\langle \psi, F_t \rangle
	= \alpha_{0,c} \cdot \mathrm{m}_{G/L}(\mathcal{R}_t) \cdot \left(1+O(e^{-t\eta_1})\right)
\]
holds as soon as \( \mathrm{m}_{G/L}(\mathcal{R}_t) \) grows exponentially, which is equivalent to \( u-\tau Y \in \mathring{C}_{\mathrm{eff}}^\vee \).
It would be interesting to determine whether this condition is sufficient to ensure validity of the asymptotic equivalent for \( \mu_{x,\tau,t}(f) \).

\subsection*{Uniform local distribution}
A very natural problem in the study of local distribution of rational points on a variety is to determine the maximal zoom factor \( \tau_X  \) such that the zooming measures \( \mu_{x,\tau,t} \) equidistribute for \emph{every} point \( x \) in \( X(\mathbb{R}) \) for every \( \tau \in (0, \tau_X) \).
Note that one always has \( \tau_X \leq \beta_X \).
In the case where \( X=\mathbb{P}^n \), it is not difficult to check that \( \tau_X = 1 \).
As explained in the introduction, it was also shown in \cite{HSS1} that if \( X \) is a non-degenerate quadric hypersurface, then \( \tau_X = \tfrac{1}{2} \).
In general, one may define the \emph{essential Diophantine exponent} of a point \( x \) in \( X \) as 
\[
\beta_{\mathrm{ess}}(x)=\sup\left\{\beta\geqslant 0:\begin{aligned}
		&\exists \text{ Zariski dense sequence } (v_i)\subset X(\mathbb{Q})\\ 
		&\text{such that }\operatorname{d}(x,v_i)\leq H(v_i)^{-\beta}.
	\end{aligned}\right\}.
\]
It seems reasonable to conjecture, at least for flag varieties of rank one, that \( \tau_X \) is equal to the essential Diophantine exponent of any rational point on \( X \).
In the case of a Grassmann variety \( X = \mathrm{Gr}_{l,d} \), this should yield \( \tau_{\mathrm{Gr}_{l,d}}=\frac{1}{\min(l,d-l)} \).

\subsection*{Geometric interpretation of constants \( \kappa_X \) and \( \alpha_X \)}

In \cite[Question~4.8]{Peyre2}, Peyre gives an interpretation of the constant \( \kappa_X \) from Theorem~\ref{th:count} in terms of the Tamagawa measure on \( X \).
It would be interesting to check by a more careful analysis of our computations, in the spirit of what is done in Borovoi-Rudnick~\cite[Theorem~4.2]{Borovoi-Rudnick}, that the expression we obtain for \( \kappa_X \) is indeed equal to Peyre's conjectural value.
In a similar vein, one should express the constant \( \alpha_X \) from Theorems~\ref{th:gld} and \ref{th:gldg} in terms of arithmetic and geometric constants related to \( X \).

\section*{Acknowledgements} 
Part of this work was done and reported during the 2024 conference \enquote{Diophantine Approximation, Dynamical Systems and Related Topics} held at Tsinghua Sanya International Mathematics Forum, whose hospitality and financial support are gratefully acknowledged.
Z.H. thanks Université Sorbonne Paris Nord for hosting multiple visits for collaboration, and thanks 
Huajie Li for helpful discussions. Z. H. was partially supported by National Key R\&D Program of China No. 2024YFA1014600.
N.S. was partially supported by the joint ANR-SNF project \enquote{Equidistribution in Number Theory} (FNS no. 10.003.145 and ANR-24-CE93-0016).

\bibliographystyle{alpha}
\bibliography{bibliography}



 
\end{document}